\documentclass[11pt]{article}

\usepackage{indentfirst,mathrsfs}
\usepackage{amsmath,amssymb,amsfonts,amsthm,mathtools}
\usepackage{array,tabularx,listings}
\usepackage{cite,url}
\usepackage[colorlinks=true,linkcolor=blue,citecolor=blue,urlcolor=blue]{hyperref}
\usepackage[capitalize,noabbrev]{cleveref}

\newtheorem{theorem}{Theorem}
\newtheorem{lemma}{Lemma}
\newtheorem{proposition}{Proposition}
\newtheorem{corollary}{Corollary}
\theoremstyle{definition}
\newtheorem{definition}{Definition}
\newtheorem{construction}{Construction}
\newtheorem{example}{Example}
\theoremstyle{remark}

\crefname{theorem}{Theorem}{Theorems}
\crefname{lemma}{Lemma}{Lemmas}
\crefname{proposition}{Proposition}{Propositions}
\crefname{corollary}{Corollary}{Corollaries}
\crefname{definition}{Definition}{Definitions}
\crefname{construction}{Construction}{Constructions}
\crefname{example}{Example}{Examples}
\crefname{remark}{Remark}{Remarks}
\crefname{section}{Section}{Sections}
\crefname{table}{Table}{Tables}

\newcommand{\F}{\mathbb{F}}
\newcommand{\Z}{\mathbb{Z}}
\newcommand{\DDT}{\operatorname{DDT}}
\newcommand{\du}{\delta}
\newcommand{\PGL}{\operatorname{PGL}}
\newcommand{\PG}{\operatorname{PG}}
\newcommand{\ord}{\operatorname{ord}}
\newcommand{\cA}{\mathcal{A}}

\begin{document}

\title{Three Infinite Classes of APN Permutations on $\Z_n$}

\author{Deng Tang\\
\small Shanghai Jiao Tong University, Shanghai, 200240, China\\dtang@foxmail.com}
\date{August 17, 2026}

\maketitle

\begin{abstract}
For any permutation of a nontrivial finite abelian group, the differential uniformity is at least two; a permutation
achieving this lower bound is called almost perfect nonlinear (APN).
In this paper, we present three infinite classes of APN permutations on the cyclic
group $\Z_n$.
First, using the orbit of a Singer cycle to label the points of the projective line cyclically, we present a class of APN permutations on $\Z_{q+1}$, where $q$ is an arbitrary odd prime power.
Each projectivity $M\in\PGL_2(q)$ that does not centralize the unique
involution $J$ of the corresponding Singer group induces such an APN
permutation.  For a fixed generator $R$ of a Singer group and a fixed base point, we obtain exactly $q(q^2-1)-2(q+1)$ projectivities satisfying this condition, and they induce distinct APN permutations.
Second, using binomials that induce permutations of the points of projective spaces over finite fields, we present a class of APN permutations on $\Z_{(3^d-1)/2}$. More generally, if a binomial $H(x)=ux^r+vx^t$ induces a permutation of the points of
the projective space $\PG(d-1,q)$ and satisfies $\gcd(r-t,q^d-1)=q-1$, then the corresponding permutation has differential uniformity at most $q-1$. For $q=3$, this immediately gives the bound $q-1=2$, and invertible ternary linearized polynomials provide an explicit infinite APN subfamily.
Third, we combine the completed reciprocal with parity and the quadratic character to present a class of APN permutations on $\Z_{2p}$, where by the completed reciprocal we mean the permutation $\iota$ of $\F_p$ defined by $\iota(0)=0$ and $\iota(x)=x^{-1}$ for $x\ne0$.
For each prime $p\ge11$, this construction defines a parametrized family $\mathcal B_{a,c,p}$, where $a\in\F_p^*$ and $c\in\F_p$. We prove that $\mathcal B_{a,c,p}$ is APN exactly when $\chi(a)=\chi(-3)=\chi(1-4a)=\chi(1-4/a)=-1$,
where $\chi$ denotes the quadratic character of $\F_p$.
For every prime $p>5$ with $p\equiv5\pmod6$, the family contains exactly $p(p+1+T_p)/8$ distinct APN permutations, where $T_p=\sum_{z\in\F_p}\chi(z(z-4)(1-4z))$ and $|T_p|\le2\sqrt p$.
Each construction also gives an infinite subfamily whose domain orders fall outside the standard order forms $r-1$, $r-2$, $r-3$, and $r-4$, where $r$ is a prime power.
Here $r-1$ is the order form of the Welch--Costas and Panario--Sakzad--Stevens--Wang classes [IEEE TIT 57(11): 7648--7657, 2011], while $r-2$, $r-3$, and $r-4$ are the order forms of the Golomb $G_2$, $G_3$, and $G_4$ classes, respectively.
To the best of our knowledge, these are the first infinite classes reported since the Panario--Sakzad--Stevens--Wang class was introduced in 2011 that provide APN permutations on $\Z_n$ for infinitely many composite orders $n$ outside the standard forms $r-1$, $r-2$, $r-3$, and $r-4$, where $r$ is a prime power.

\end{abstract}


\section{Introduction}\label{S:Introduction}

Let $G$ be a nontrivial finite abelian group written additively, and let $F$ be a permutation of $G$. The differential uniformity of $F$ is the maximum number of
solutions $x\in G$ to an equation $F(x+a)-F(x)=b$,
where $a$ ranges over $G\setminus\{0\}$ and $b$ ranges over $G$ (see \cref{D:APN} for details). A perfect nonlinear (PN) function has differential
uniformity one. For any fixed nonzero input difference $a$, the injectivity of $F$ implies that $F(x+a)-F(x)\ne 0$ for every $x\in G$.
Thus, the $|G|$ resulting output differences take values among only the $|G|-1$ nonzero elements of $G$.
The pigeonhole principle therefore implies that some output difference must occur at least twice. Consequently,
every permutation of $G$ has differential uniformity at least two, and hence no PN permutation exists
on $G$~\cite{DrakakisGowMcGuireDAM2009}. A permutation achieving this lower bound is called almost perfect nonlinear (APN).
Differential uniformity is a classical cryptographic criterion, and its
definition on finite abelian groups generalizes the usual definition on
vector spaces over finite fields~\cite{NybergEC1994}.

In this paper, we write $\Z_n=\Z/n\Z$ and consider its cyclic additive group; its ring structure is used only when multiplication or polynomial representations are
involved.  The truth table of a permutation may be regarded as the cyclic
sequence $[F(0),F(1),\ldots,F(n-1)]$.  For every nonzero cyclic shift, the APN
property requires each entry in the difference of the shifted and unshifted
sequences to occur at most twice (see \cref{D:APN}).  The same condition
occurs in the study of Costas permutations, distinct-difference
configurations, and Sidon-type sets.  More precisely, in the graph of $F$ in
$\Z_n\times\Z_n$, every difference with nonzero first coordinate has at most
two ordered representations; the multiplicity of $(u,v)$ is
$\DDT_F(u,v)$ for $u\ne0$.  Over elementary abelian $2$-groups, the
corresponding bound on every nonzero graph difference establishes the Sidon-set
characterization of an APN function~\cite{MihailaThornburghCC2026}.  Related
Sidon spaces also occur in cyclic orbit codes~\cite{RothRavivTamoTIT2018}.

Permutations of residue classes have appeared directly in symmetric
cryptography.  The two nonlinear byte substitutions of \textsf{SAFER K-64}
are obtained from exponentiation and discrete logarithms to the base $45$
modulo $257$, with the exceptional representatives completed so that the
functions permute the $256$ byte values~\cite{MasseyFSE1994}.  When the byte values
are identified with $\Z_{256}$, the exponentiation permutation is closely related to
the Welch--Costas APN permutation for the prime $257$
\cite{DrakakisGowMcGuireDAM2009,FreitasUFSC2019}.
Rivest characterized all integer-coefficient polynomials that induce permutations of $\Z_{2^w}$ and observed that
the \textsf{RC6} block cipher uses $x\mapsto x(2x+1)$ as a permutation of $w$-bit words~\cite{RivestFFA2001}.
 Klimov and Shamir constructed invertible transformations on
$n$-bit words by combining arithmetic modulo $2^n$ with Boolean operations
\cite{KlimovShamirCHES2002}.  Computational searches for APN permutations on
groups such as $\Z_2\times\Z_{2^k}$ were reported in
\cite{FreitasUFSC2019}.  These examples highlight two distinct requirements:
efficient invertibility is useful in a cryptographic implementation, whereas
low differential uniformity requires a separate analysis.

Fully homomorphic encryption (FHE), zero-knowledge (ZK) proofs, and secure
multiparty computation (MPC) have renewed interest in cryptographic primitives
expressed by ring arithmetic.  Arithmetic modulo $2^{32}$ or $2^{64}$ agrees
with the word operations supported by common processors.  Examples of
symmetric primitives designed for such arithmetic include \textsf{Elisabeth}
over $\Z_{16}$ and \textsf{Rubato} with composite moduli
\cite{CosseronHoffmannMeauxStandaertAC2022,HaKimLeeLeeSonEC2022}.  The
lattice-based pseudorandom number generator \textsf{LEAP} operates over the
quotient ring $\Z_q[X]/(X^d+1)$ and provides another example in which modular
ring arithmetic is part of a symmetric primitive
\cite{ZhangLuLiuYinWangTOSC2025}.  Mixing arithmetic and Boolean operations,
or using rings with zero divisors, can invalidate algebraic identities that
hold over finite fields.  This does not by itself imply security; it requires
cryptanalysis in the actual ambient ring
\cite{GrassiAyalaHovdOygardenRaddumWangC2023}.
Ring arithmetic is also used directly in proof systems and MPC.  Arithmetic
secret sharing over Galois rings gives MPC over $\Z_{p^k}$
\cite{AbspoelEtAlAC2020}, and dishonest-majority MPC over $\Z_{2^k}$ can be
constructed by working in Galois ring extensions
\cite{EscuderoXingYuanC2022}.  \textsf{Rinocchio} proves computations over
rings such as $\Z_{2^{64}}$~\cite{GaneshNitulescuSoriaVazquezJC2023}.
More recent work gave efficient ZK protocols over $\Z_{2^k}$ via Galois rings
\cite{LinXingYaoC2024} and polynomial commitments leading to publicly
verifiable SNARKs for $\Z_{2^k}$ circuits
\cite{JiaLiXingYaoYuanC2025}.  These results reduce the need to translate
word-level computations into arithmetic circuits over an unrelated field.

The NIST First Call for Multi-Party Threshold Schemes solicits distributed implementations in which a cryptographic operation is
performed while the private or secret key remains shared among several parties. Its scope includes signatures,
public-key and symmetric-key primitives, key generation, FHE, ZK proofs, and auxiliary components~\cite{BrandaoPeraltaNIST2026}.
Together with the availability of MPC, secret-sharing, and ZK protocols over residue class rings, this broad scope suggests that arithmetic
over $\Z_n$ could be investigated as a potential computational domain for future threshold-friendly symmetric components or auxiliary operations~\cite{AbspoelEtAlAC2020,EscuderoXingYuanC2022,LinXingYaoC2024,JiaLiXingYaoYuanC2025}. The APN permutations constructed in this paper provide combinatorial candidates for such further investigation.
The NIST standards ML-KEM and ML-DSA operate over polynomial quotient rings with coefficients in $\Z_q$, rather than over the cyclic additive group $\Z_n$~\cite{NISTFIPS2032024,NISTFIPS2042024}. Related work has developed lattice-based threshold signatures, distributed key generation, and threshold key-encapsulation mechanisms
in polynomial-ring settings~\cite{DelPinoEtAlEC2024,EspitauNiotPrestC2024,LapihaPrestAC2025}. These standards
and constructions illustrate the broader role of modular ring arithmetic, but they do not provide a direct application of
permutations on $\Z_n$. The present paper establishes combinatorial properties only; it proposes neither a cryptographic primitive
nor a threshold protocol and does not evaluate implementation or circuit costs.

To date, several infinite classes of  APN permutations on $\Z_n$ are known.  The
Welch--Costas class is obtained from exponentiation in the multiplicative
group of a finite field~\cite{GolombTaylorPIEEE1984}.  Drakakis, Gow, and
McGuire subsequently gave a common treatment of APN permutations on $\Z_n$,
including the Welch and Golomb classes
\cite{DrakakisGowMcGuireDAM2009}.  Panario, Sakzad, Stevens, and Wang then used
finite field permutations and discrete logarithms to obtain permutations with
optimum ambiguity~\cite{PanarioSakzadStevensWangTIT2011} (this class will be
called the Panario--Sakzad--Stevens--Wang class in the sequel).  In
particular, Panario et al. called the completed rational permutation in their
Theorem~15 a M\"obius function.  That theorem applies exponentiation and
discrete logarithms to this function.  Hence neither fractional linear transformations
nor the use of exponents and discrete logarithms to label finite field
elements is new.  The natural orders of these earlier constructions are inherited from a full
multiplicative group or from deleting a small number of field elements.
Related circular Costas maps and projective group actions were considered in
\cite{RubioTorresARXIV2022,FengLiZhouARXIV2020}.
In 2013, Yu and Wang proved that a permutation induced by a polynomial with integer coefficients over composite $\Z_n$ cannot be APN and
identified the construction of more general forms as an open problem~\cite{YuWangDAM2013}.
Other results
gave differential uniformity at most eight on
$\Z_{2p}$~\cite{MishraGuptaGaurFFA2023} and at most three on the prime ring
$\Z_p=\F_p$~\cite{GuptaMishraGaurDCC2025}, but did not provide new infinite
APN classes on cyclic groups of composite order.  Recent APN permutations over prime
fields~\cite{BudaghyanPalDCC2025} concern the field case, whereas the
classification over the Galois ring $\operatorname{GR}(4,2)$ concerns the
noncyclic additive group $\Z_4^2$~\cite{RonjonSandribCC2026}.
To the best of our knowledge, the classes presented in this paper are the
first infinite classes reported since the Panario--Sakzad--Stevens--Wang
class was introduced in 2011 that provide APN permutations on $\Z_n$ for
infinitely many composite orders $n$ outside the standard forms
$r-1,r-2,r-3,r-4$, where $r$ is a prime power.

In this paper, we construct three infinite classes of APN permutations on the
cyclic group $\Z_n$.  The main contributions are summarized as follows.
\begin{itemize}
\item The first construction, described in \cref{Con:singer}, uses the orbit of a Singer cycle to label the projective line cyclically and yields a class of APN permutations on $\Z_{q+1}$, where $q$ is an arbitrary odd prime power. By \cref{T:singer-apn}, each projectivity $M\in\PGL_2(q)$ that does not centralize the unique
involution $J$ of the corresponding Singer group induces such an APN
permutation. The proof reduces each difference equation to the fixed-point equation of a nonidentity M\"obius transformation. \Cref{P:singer-count} shows that, for a fixed Singer cycle and a fixed base point, exactly $q(q^2-1)-2(q+1)$ projectivities satisfy this condition and induce distinct APN permutations.

\item The second construction, given in \cref{Con:binomial}, starts with binomials that induce permutations of the points of projective spaces over finite fields and gives a class of APN permutations on $\Z_{(3^d-1)/2}$. More generally, \cref{T:exponent-condition} shows that if a binomial
$H(x)=u x^r+v x^t$
induces a permutation of the points of projective space $\PG(d-1,q)$ and satisfies
$\gcd(r-t,q^d-1)=q-1$, then the corresponding permutation has differential uniformity at most $q-1$. For $q=3$, this upper bound becomes two, so the induced permutation is APN. Invertible ternary linearized polynomials yield the explicit infinite APN subfamily given in \cref{T:ternary-apn}.

\item For an arbitrary prime $p\ge11$, the third construction in \cref{Con:reciprocal} combines the completed reciprocal with parity and the quadratic character to define a parametrized family $\mathcal B_{a,c,p}$ of functions on $\Z_{2p}$, where $a\in\F_p^*$ is a multiplier parameter and $c\in\F_p$ is a translation parameter. By \cref{T:reciprocal-apn}, the function $\mathcal B_{a,c,p}$ is an APN permutation if and only if
$\chi(a)=\chi(-3)=\chi(1-4a)=\chi(1-4/a)=-1$,
where $\chi$ denotes the quadratic character of $\F_p$. In particular, this condition is independent of $c$. \Cref{T:parameter-count,C:reciprocal-infinite} show that, for every prime $p>5$ satisfying $p\equiv5\pmod6$, the family contains exactly
$p(p+1+T_p)/8$
distinct APN permutations, where $T_p=\sum_{z\in\F_p}\chi\bigl(z(z-4)(1-4z)\bigr)$ and
$|T_p|\le2\sqrt p$.
\end{itemize}

Each construction also gives an infinite subfamily whose domain orders fall
outside the standard forms $r-1,r-2,r-3,r-4$, where $r$ is a prime power.
Here $r-1$ is the order form of the Welch--Costas and
Panario--Sakzad--Stevens--Wang classes, while $r-2,r-3,r-4$ are the order
forms of the Golomb $G_2,G_3,G_4$ classes, respectively.
Relabeling inputs or outputs and taking compositional inverses do not change
the domain order, so these orders distinguish the new permutations from those earlier
classes under these operations.  This comparison neither decides equivalence when two
constructions have the same order nor proves disjointness under a broader
notion of equivalence.
Table~\ref{Tab:intro-comparison} compares the order forms of the known classes
with those of the three classes presented in this paper.  Here $r$ and $q$
denote prime powers, and $p$ denotes a prime.  The additional hypotheses are
given in the cited results and in the subsequent sections.

\begin{table}[htbp]
\centering
\caption{Known and new infinite classes of APN permutations on $\Z_n$}
\label{Tab:intro-comparison}
\footnotesize
\begin{tabular}{|c|c|c|}
\hline
\textbf{Class} & \textbf{Order $n$ and parameter range} & \textbf{Reference} \\
\hline
Welch--Costas
& $r-1$, with $r$ a prime power
& \cite{GolombTaylorPIEEE1984,DrakakisGowMcGuireDAM2009} \\
\hline
Golomb $G_2,G_3,G_4$
& $r-2,r-3,r-4$, with $r$ a prime power
& \cite{GolombTaylorPIEEE1984,DrakakisGowMcGuireDAM2009} \\
\hline
Panario--Sakzad--Stevens--Wang
& $r-1$, with $r$ a prime power
& \cite{PanarioSakzadStevensWangTIT2011} \\
\hline
Singer cycle
& $q+1$, with $q$ an odd prime power
& \cref{T:singer-apn} \\
\hline
Binomial over $\F_{3^d}$
& $(3^d-1)/2$, with $d\ge2$
& \cref{T:ternary-apn} \\
\hline
Completed reciprocal
& $2p$, with $p>5$ prime and $p\equiv5\pmod6$
& \cref{C:reciprocal-infinite} \\
\hline
\end{tabular}
\end{table}

The remainder of this paper is organized as follows.  In
Section~\ref{S:Preliminaries}, the necessary notation and preliminary results
are reviewed.  The permutations induced by Singer cycles are presented in
Section~\ref{S:Singer}.  The construction using binomials over
$\F_{3^d}$ is developed in Section~\ref{S:Binomial}, and the permutations obtained
from completed reciprocals are investigated in Section~\ref{S:Reciprocal}.
A comparison with the standard constructions is given in
Section~\ref{S:Comparison}.  Finally,
Section~\ref{S:Conclusion} concludes the paper.  Representative truth tables
of moderate orders are recorded in Section~\ref{S:Appendix}.

\section{Preliminaries}\label{S:Preliminaries}

In this section, we introduce the notation and recall the preliminary results
used in the three constructions.  We first consider cyclic differences and
some elementary number-theoretic tools.  We then review the projective
geometry and finite field notions required in
Sections~\ref{S:Singer} and~\ref{S:Binomial}.

\subsection{Differential uniformity, difference distribution tables, and differential spectra}

Differential cryptanalysis, introduced by Biham and Shamir, studies how
input differences propagate to output differences through a cryptographic
transformation~\cite{BihamShamirJC1991}.  Motivated by this attack, Nyberg
introduced differential uniformity as the maximum, over all nonzero input
differences and all output differences, of the number of solutions to the
corresponding difference equation~\cite{NybergEC1994}.  Almost perfect nonlinear permutations were
also studied by Beth and Ding~\cite{BethDingEC1994}.  We use the same
notion on the additive group $\Z_n$.

\begin{definition}\label{D:APN}
Let $F:\Z_n\rightarrow\Z_n$ be a function.  Its \emph{differential
uniformity} is
$\du(F)=\max_{u\ne0,\,v\in\Z_n}\#\{x\in\Z_n:F(x+u)-F(x)=v\}$.
The function is called
\emph{perfect nonlinear} (PN) if $\du(F)=1$ and \emph{almost perfect
nonlinear} (APN) if $\du(F)=2$.
All additions and subtractions are taken in the cyclic group $\Z_n$.
Throughout this paper, APN is used in this group-theoretic sense.
\end{definition}

The numbers occurring in \cref{D:APN} are recorded in the
\emph{differential distribution table} (DDT) of $F$.  Its entry in row $u$
and column $v$ is
$\DDT_F(u,v)=\#\{x\in\Z_n:F(x+u)-F(x)=v\}$.
For a permutation $F:\Z_n\rightarrow\Z_n$, its \emph{differential spectrum} is the
multiset $ [\DDT_F(u,v)\,|\,u\in\Z_n\setminus\{0\},\ v\in\Z_n]$.
The same multiset is specified by
\begin{equation}
 \mathcal S(F)=(A_0,A_1,\ldots,A_n),                  \label{E:spectrum}
\end{equation}
where $ A_j=\#\{(u,v)\in(\Z_n\setminus\{0\})\times\Z_n:  \DDT_F(u,v)=j\}$.
Thus $A_j$ is the multiplicity of $j$ in the differential spectrum.
By counting the DDT entries and their total multiplicity, we have $\sum_{j=0}^n A_j=n(n-1)$ and $\sum_{j=0}^n jA_j=n(n-1)$.
For an APN permutation, only $A_0,A_1,A_2$ can be nonzero.
For every fixed $u\ne0$,
$\sum_{v\in\Z_n}\DDT_F(u,v)=n$ and $\DDT_F(u,0)=0$, since $F$ is
injective.
Thus, the $n$ derivative values take at most $n-1$ possible nonzero
values, and therefore $\du(F)\ge2$. Hence, APN permutations attain the
lowest differential uniformity possible for a permutation. At the S-box
level, this minimizes the maximum probability of a differential transition
with a nonzero input difference. The differential resistance of a complete
cipher, however, also depends on its diffusion layer, round structure, and
the resulting number of active S-boxes~\cite{BihamShamirJC1991}.
Input and output group automorphisms, translations, and compositional
inversion relabel or transpose the DDT entries.  They therefore preserve
both differential uniformity and $\mathcal S(F)$ in \eqref{E:spectrum}.  In
particular, $\DDT_{F^{-1}}(v,u)=\DDT_F(u,v)$.  The row with input difference
zero and the column with output difference zero both contain one entry equal
to $n$ and $n-1$ zero entries.  It follows that transposition also preserves
the spectrum after the zero-input row is omitted.  If two constructions have
different domain orders, no relabeling of their inputs and outputs can make
them equivalent.  We use this fact for the infinite families below, and we
use the differential spectrum to check the examples in the appendix.

The small-order tables in Sections~\ref{S:Singer}--\ref{S:Reciprocal} summarize
the number of distinct differential spectra and the number of permutations having each
spectrum.  Distinct spectra imply inequivalence under the operations above;
the converse need not hold.

\subsection{Number-theoretic preliminaries}

We recall the number-theoretic notions that are used in the
proofs.   For any prime $r$, the $r$-adic valuation $v_r(n)$ is the largest exponent
$e$ such that $r^e$ divides $n$.
We repeatedly use the following observation.  If $m>r$ and
$v_r(m)=1$, then $m/r>1$ has a prime divisor different from $r$, so $m$ has
at least two distinct prime divisors and is not a prime power.
We use the standard lifting-the-exponent formula
\begin{equation}
 v_2(x^m-1)=v_2(x-1)+v_2(x+1)+v_2(m)-1              \label{E:lte}
\end{equation}
when $x$ is odd and $m$ is even.
Dirichlet's theorem on primes in arithmetic progressions states that every
progression $r\bmod m$ with $\gcd(r,m)=1$ contains infinitely many primes.
Let $p$ be an odd prime.  The \emph{Legendre symbol}, or quadratic character,
is defined as
\begin{equation*}
 \chi(x)=\left(\frac{x}{p}\right)=
 \begin{cases}
  0,&x=0,\\
  1,&x\text{ is a nonzero square in }\F_p,\\
 -1,&x\text{ is a nonsquare in }\F_p.
 \end{cases}
\end{equation*}
 We use the multiplicativity
$\chi(xy)=\chi(x)\chi(y)$ and quadratic reciprocity.  In particular, for
$p>3$, $ \chi(-3)=-1 $ if and only if $ p\equiv5\pmod6$.
We also use the following elementary quadratic character sum.
\begin{lemma}\label{L:quadratic-character-sum}
If $\lambda\in\F_p^*$ and $r,s\in\F_p$ are distinct, then
\begin{equation*}
 \sum_{z\in\F_p}\chi\bigl(\lambda(z-r)(z-s)\bigr)=-\chi(\lambda).
\end{equation*}
\end{lemma}

\begin{proof}
Let $m=(r+s)/2$ and $d=(r-s)/2$, so $d\ne0$ and
$(z-r)(z-s)=(z-m)^2-d^2$.  Let $N$ be the number of pairs
$(u,v)\in\F_p^2$ satisfying $v^2=u^2-d^2$.  Note that $ N=\sum_{u\in\F_p}\bigl(1+\chi(u^2-d^2)\bigr)$
and  $(u-v)(u+v)=d^2$.  For each
$w\in\F_p^*$ there is exactly one pair $(u,v)$ with $u-v=w$, namely
$u=(w+d^2/w)/2$ and $v=(d^2/w-w)/2$.  Hence $N=p-1$.  Comparing the
two expressions gives $\sum_u\chi(u^2-d^2)=-1$.  Multiplication by $\lambda$ and
the multiplicativity of $\chi$ prove the result.
\end{proof}

We also use the following special case of the Weil bound for character sums.
If $f\in\F_p[X]$ is a square-free cubic, then
\begin{equation}
 \left|\sum_{z\in\F_p}\chi(f(z))\right|\le2\sqrt p. \label{E:weil}
\end{equation}
Indeed, the smooth projective curve $Y^2=f(X)$ has genus one and
$p+1+\sum_{z\in\F_p}\chi(f(z))$ rational points; the Hasse--Weil bound gives
\eqref{E:weil}.
For the primes $p\equiv5\pmod6$ considered below, this estimate guarantees
the existence of multipliers satisfying the conditions in the third
construction.  We refer to
\cite{IrelandRosenGTM1990,LidlNiederreiterCUP1997}
for these facts.

\subsection{Projective spaces, projectivities, and Singer cycles}

The first two constructions will use projective geometry, so we recall the required
notions in some detail.  Let $V=\F_q^d$.  A point of the projective space
$\PG(d-1,q)$ is a one-dimensional vector subspace of $V$.  Equivalently, a
nonzero vector $(x_0,\ldots,x_{d-1})$ represents the homogeneous point $ [x_0:\cdots:x_{d-1}]$,
and two vectors represent the same point precisely when they differ by a
nonzero scalar in $\F_q$.  Therefore $ |\PG(d-1,q)|=\frac{q^d-1}{q-1}$.
For $d=2$, the projective line has the affine chart $ \mathbb P^1(\F_q)=\{[x:1]:x\in\F_q\}\cup\{[1:0]\}=\F_q\cup\{\infty\}$.
Let $\operatorname{GL}_d(q)$ denote the general linear group of invertible $d\times d$ matrices over $\F_q$. Every matrix
in $\operatorname{GL}_d(q)$ defines an invertible $\F_q$-linear transformation of $V$ and therefore acts on its one-dimensional subspaces, namely,
the points of $\PG(d-1,q)$. Since every nonzero scalar matrix $\lambda I_d$ acts trivially on these points, the resulting projective linear group
is $\PGL_d(q)=\operatorname{GL}_d(q)/\{\lambda I_d:\lambda\in\F_q^*\}$. On the projective line $\mathbb P^1(\F_q)=\F_q\cup\{\infty\}$,
a matrix $M=\begin{pmatrix}a&b\\c&d\end{pmatrix}\in\operatorname{GL}_2(q)$, where $ad-bc\ne0$,
induces the M\"obius transformation $x\mapsto (ax+b)/(cx+d)$. For $x\in\F_q$ with $cx+d=0$, its image is $\infty$; moreover, $M(\infty)=a/c$ if $c\ne0$, and $M(\infty)=\infty$ if $c=0$.
\begin{lemma}\label{L:two-fixed-points}
A nonidentity M\"obius transformation has at most two fixed points on
$\mathbb P^1(\F_q)$.
\end{lemma}

\begin{proof}
For a point $[X:Y]$, the fixed-point condition is the homogeneous equation
\begin{equation*}
 -cX^2+(a-d)XY+bY^2=0.
\end{equation*}
It is a nonzero quadratic form unless the matrix is scalar, which would give
the identity in $\PGL_2(q)$.  A nonzero homogeneous quadratic has at most two
zeros on the projective line.  This includes $[1:0]=\infty$; in the affine
chart the equation reduces to $cx^2+(d-a)x-b=0$, with the linear case
$c=0$ included.
\end{proof}

A \emph{Singer cycle} of $\PG(d-1,q)$ is a cyclic group that acts regularly
on the projective points; the same term is also used for a generator of this
group.  Here ``regularly'' means that for every ordered pair of points there is
exactly one group element taking the first point to the second.  To construct
one, identify $V$ with $\F_{q^d}$ as an
$\F_q$-vector space.  Multiplication by a primitive element $\alpha$ permutes
the one-dimensional subspaces.  Scalar multiplications by $\F_q^{*}$ are
projectively trivial, and hence the induced cycle has order $(q^d-1)/(q-1)$.
For $d=2$, a Singer group is a cyclic subgroup $C$ of $\PGL_2(q)$ of order $q+1$; such a subgroup is also called
a \emph{nonsplit torus}. Every nonidentity element of $C$ is diagonalizable over $\F_{q^2}$ but not over $\F_q$, and all such elements
have the same pair of conjugate eigenlines over $\F_{q^2}$. Conversely, this pair of eigenlines uniquely determines $C$. The normalizer of
$C$ in $\PGL_2(q)$ is isomorphic to $C_{q+1}\rtimes C_2$, where the nonidentity element of $C_2$ acts on $C_{q+1}$ by inversion; in particular,
the normalizer has order $2(q+1)$. For these standard structural facts and further background on finite projective geometry, see
\cite{HirschfeldOUP1998,LidlNiederreiterCUP1997}.
For a group $G$, an element $x\in G$, and a subgroup $C\le G$, we use the
standard notation $C_G(x)=\{g\in G:gx=xg\}$ and $N_G(C)=\{g\in G:gCg^{-1}=C\}$
for the centralizer of $x$ and the normalizer of $C$, respectively.

\subsection{Labeling the points of \texorpdfstring{$\PG(d-1,q)$}{PG(d-1,q)} by exponents}

Let $K=\F_{q^d}$ and $\alpha$ be a primitive element of $K^{*}=K\setminus \{0\}$.
Put $n_q=(q^d-1)/(q-1)$.  The quotient $K^*/\F_q^*$ of nonzero vectors by
$\F_q^*$-scalars is the point set of $\PG(d-1,q)$.  In what follows, $[x]$
denotes the point $\F_qx$.  Since $\alpha^{n_q}\in\F_q^*$ and
$\alpha\F_q^*$ generates $K^*/\F_q^*$, the correspondence $i\mapsto[\alpha^i]$ is a
bijection from $\Z_{n_q}$ to the point set of $\PG(d-1,q)$.  This bijection
transfers permutations of the projective points to permutations of $\Z_{n_q}$.
If a function
$H:K^*\to K^*$ satisfies
$H(\lambda x)\in\F_q^*H(x)$ for all $\lambda\in\F_q^*$, then
$[x]\mapsto[H(x)]$ is well defined on these projective points.  In particular,
the binomial $u x^r+v x^t$ has this property whenever
$r\equiv t\pmod {q-1}$.  An invertible $\F_q$-linear transformation $L$ also induces a
projective transformation because $L(\lambda x)=\lambda L(x)$.  A polynomial
$L(x)=\sum_j c_jx^{q^j}$ is called a $q$-\emph{linearized polynomial}; it
represents an $\F_q$-linear transformation on every extension field containing its
coefficients.

\subsection{Completed reciprocals}

For an odd prime $p$, the \emph{completed reciprocal} is the permutation $\iota$ of $\F_p$ defined
by $\iota(0)=0$ and $\iota(x)=x^{-1}$ for $x\ne0$. By Fermat's little theorem, $\iota(x)=x^{p-2}$ for
every $x\in\F_p$, and $\iota$ is an involution. Projective inversion $x\mapsto x^{-1}$ extends naturally to a M\"obius transformation on $\mathbb P^1(\F_p)$ that interchanges $0$ and $\infty$; the completed reciprocal instead remains on $\F_p$ and fixes $0$.
Let $\chi$ denote the quadratic character of $\F_p$. Since inversion preserves quadratic character, $\chi(\iota(x))=\chi(x)$. We denote the indicator of the nonzero squares by $\rho$, so that $\rho(x)=1$ if $\chi(x)=1$ and $\rho(x)=0$ otherwise; in particular, $\rho(0)=0$.

\section{APN Permutations from Singer Cycles}\label{S:Singer}
The first construction uses the orbit of a Singer cycle to label the
projective line by the elements of $\Z_{q+1}$.  Under the labeling
$i\mapsto R^iP_0$, adding $a$ to a label corresponds to applying $R^a$ to the
projective point.  A differential equation on $\Z_{q+1}$ can therefore be
translated into a fixed-point equation for a projectivity.
In  \cite[Theorem~15]{PanarioSakzadStevensWangTIT2011}, Panario et al.
called the completed rational permutation a M\"obius function and combined it
with exponentiation and discrete logarithms on a cyclic group of order $q-1$.
Here the cyclic
group generated by a Singer cycle acts on all $q+1$ projective points, and
the construction imposes the condition $MJ\ne JM$, which
ensures that the projectivity arising from every differential equation is not
the identity.

\subsection{The construction}

\begin{construction}\label{Con:singer}
Let $q$ be an arbitrary odd prime power and put $n=q+1$. Let
$C\le\PGL_2(q)$ be any Singer group. Fix a generator $R$ of $C$ and a base
point $P_0\in\mathbb P^1(\F_q)$. Here $R^iP_0$ denotes the image of the
projective point $P_0$ under $R^i$, not a chosen vector representative.
Since $C$ acts regularly on $\mathbb P^1(\F_q)$, the function
$\exp_{R,P_0}\colon\Z_n\to\mathbb P^1(\F_q)$ defined by
$\exp_{R,P_0}(i)=R^iP_0$ is a bijection. Denote its inverse by
$\log_{R,P_0}$, and let $J=R^{n/2}$ be the unique element of order two in
$C$. For each $M\in\PGL_2(q)$, define
\begin{equation}
F_M=\log_{R,P_0}\circ M\circ\exp_{R,P_0}.
\label{E:singer-induced}
\end{equation}
\end{construction}

As a composition of bijections, $F_M$ is a permutation of $\Z_n$.
Moreover, $F_M(i)=j$ if and only if $M(R^iP_0)=R^jP_0$.
For an explicit realization of $R$, identify $\F_q^2$ with
$K=\F_{q^2}$, choose a primitive element $\alpha\in K^*$, and write
$\alpha^2=t\alpha+u$ with $t,u\in\F_q$.  Relative to the basis
$(1,\alpha)$, multiplication by $\alpha$ is represented by
\begin{equation*}
 R=\begin{pmatrix}0&u\\1&t\end{pmatrix}.
\end{equation*}
Here $u\ne0$, and the projective class of $R$ has order $q+1$ and acts
regularly on $\mathbb P^1(\F_q)$.  A change of basis conjugates $R$ and its
Singer group.  The values of $F_M$ are obtained by listing the orbit
$P_0,RP_0,\ldots,R^qP_0$.

The following lemma will be used to prove that the projectivity obtained from
each nonzero input and output difference is not the identity.

\begin{lemma}\label{L:singer-intersection}
For every $M\in\PGL_2(q)$, the equality
$C\cap MCM^{-1}=\{1\}$ holds if and only if $MJ\ne JM$.
\end{lemma}

\begin{proof}
If $MJ=JM$, then $J=MJM^{-1}$ is a nonidentity element of the intersection.
Conversely, suppose that $1\ne X\in C\cap MCM^{-1}$.  If $\ord(X)$ is even,
the unique involution in $\langle X\rangle$ is the unique involution in each
of the cyclic groups $C$ and $MCM^{-1}$; it is therefore both $J$ and
$MJM^{-1}$, so
$MJ=JM$.  If $\ord(X)$ is odd, then $X$, as a nonidentity element of the
Singer group $C$, has two conjugate eigenlines over $\F_{q^2}$.
The pair of conjugate eigenlines of $X$ determines the unique nonsplit torus
containing $X$. This torus is the centralizer of $X$ and has order $q+1$.
Since $C$ is abelian and contains
$X$, we have $C\le C_{\PGL_2(q)}(X)$.  The same argument gives
$MCM^{-1}\le C_{\PGL_2(q)}(X)$. Since $C$, $MCM^{-1}$, and
$C_{\PGL_2(q)}(X)$ all have order $q+1$, the three groups coincide. The
unique elements of order two in $C$ and $MCM^{-1}$ are $J$ and
$MJM^{-1}$, respectively. Hence $MJM^{-1}=J$, or equivalently, $MJ=JM$.
For the standard centralizer facts used here, see
\cite{HirschfeldOUP1998,LidlNiederreiterCUP1997}.

\end{proof}

\subsection{The APN property}

\begin{theorem}\label{T:singer-apn}
Let $C=\langle R\rangle$ and $J=R^{(q+1)/2}$ be as above.  If
$M\in\PGL_2(q)$ satisfies $MJ\ne JM$, then the permutation $F_M$ defined in
\eqref{E:singer-induced} is APN on $\Z_{q+1}$.  Its compositional inverse is
obtained by the same construction.
\end{theorem}

\begin{proof}
Let $n=q+1$.  According to the definition of differential uniformity in
\cref{D:APN}, we need to prove that the differential equation
\begin{equation}
 F_M(x+a)-F_M(x)=b                                    \label{E:singer-difference}
\end{equation}
has at most two solutions $x\in\Z_n$ for every
$a\in\Z_n\setminus\{0\}$ and $b\in\Z_n$.  Since $F_M$ is a permutation,
$\DDT_{F_M}(a,0)=\#\{x:F_M(x+a)=F_M(x)\}=0$ for every $a\ne0$, because
injectivity would otherwise imply $x+a=x$.  It remains to fix
$a,b\ne0\pmod n$, and write $P=R^xP_0$.  In particular,
$R^a\ne1$ and $R^b\ne1$.  By \eqref{E:singer-induced}, equation
\eqref{E:singer-difference} is equivalent to
\begin{equation*}
 M(R^{x+a}P_0)=R^bM(R^xP_0).
\end{equation*}
Indeed, if $j=F_M(x)$, then $M(R^xP_0)=R^jP_0$, while
\eqref{E:singer-difference} gives $F_M(x+a)=j+b$ and hence
$M(R^{x+a}P_0)=R^{j+b}P_0=R^bM(R^xP_0)$.  The converse follows by applying
$\log_{R,P_0}$ to the same equality, which gives
$F_M(x+a)=F_M(x)+b$.  With $P=R^xP_0$, this equation is
$MR^aP=R^bMP$, or equivalently $T_{a,b}P=P$ for
$T_{a,b}=M^{-1}R^{-b}MR^a$.
If $T_{a,b}=1$, then $M^{-1}R^{-b}MR^a=1$, or equivalently,
$MR^aM^{-1}=R^b$. Since $R^a\in C$, the left side belongs to $MCM^{-1}$,
whereas $R^b\in C$. Thus $R^b\in C\cap MCM^{-1}=\{1\}$ by
\cref{L:singer-intersection}, contradicting $R^b\ne1$, which follows from
$b\ne0\pmod n$ and $\ord(R)=n$. Therefore
$T_{a,b}$ is nonidentity.
By~\cref{L:two-fixed-points}, it fixes at most two projective points.  Since
$\exp_{R,P_0}$ is a bijection, each such point corresponds to a unique
$x\in\Z_n$.  Thus \eqref{E:singer-difference} has at most two solutions.
Since every
permutation has differential uniformity at least two, $F_M$ is APN.
Finally, \eqref{E:singer-induced} and
$\log_{R,P_0}=\exp_{R,P_0}^{-1}$ give
\begin{equation*}
 F_M^{-1}=(\log_{R,P_0}\circ M\circ\exp_{R,P_0})^{-1}
 =\log_{R,P_0}\circ M^{-1}\circ\exp_{R,P_0}=F_{M^{-1}}.
\end{equation*}
Moreover, $M^{-1}J=JM^{-1}$ if and only if $MJ=JM$.
Hence $M^{-1}J\ne JM^{-1}$, and applying the same argument with $M$
replaced by $M^{-1}$ shows that $F_{M^{-1}}$ is APN.
\end{proof}

\subsection{Parameters and orders}

\begin{example}\label{Ex:singer-small}
Let $f(X)=X^2-X-4\in\F_7[X]$ and
$K=\F_7[X]/(f(X))$, and denote the residue class of $X$ by $\alpha$.
The discriminant of $f$ is $3$, which is a nonsquare in $\F_7$, and
$\alpha^{24}=-1$ and $\alpha^{16}=2$.  Thus $f$ is irreducible and
$\alpha$ has order $48$.  In the ordered basis $(1,\alpha)$, multiplication
by $\alpha$ is represented by
\begin{equation*}
 R=\begin{pmatrix}0&4\\1&1\end{pmatrix},\qquad
 M=\begin{pmatrix}1&1\\0&1\end{pmatrix},~~ \text{and}~~P_0=\infty.
\end{equation*}
Put $e_0=(1,0)^{\mathsf T}$ and $P_i=[R^ie_0]=[\alpha^i]$.  Direct iteration
on $\mathbb P^1(\F_7)$ gives
\begin{equation*}
 \infty\longmapsto0\longmapsto4\longmapsto5\longmapsto3
 \longmapsto1\longmapsto2\longmapsto6\longmapsto\infty.
\end{equation*}
Thus the projective class of $R$ generates a Singer group, and this orbit lists
all projective points.
Moreover, note that
$$J=R^4=\left(\begin{matrix}6&1\\2&1\end{matrix}\right)$$
does not commute with $M$ in $\PGL_2(7)$.  The induced permutation is
\begin{equation*}
 [F_M(0),\ldots,F_M(7)]=[0,5,3,7,2,6,4,1].
\end{equation*}
Equivalently, its entries are obtained by taking the unique $j$ such that
$MR^ie_0\in\F_7^*R^je_0$.
The permutation is APN, in agreement with \cref{T:singer-apn}, and its
differential spectrum is $(A_0,A_1,A_2)=(24,8,24)$.
\end{example}

\begin{proposition}\label{P:singer-count}
Let $q$ be an odd prime power, and fix a Singer group $C=\langle R\rangle$
and a base point $P_0\in\mathbb P^1(\F_q)$ as in \cref{Con:singer}. Then
exactly
\begin{equation}
q(q^2-1)-2(q+1)
\label{E:singer-count}
\end{equation}
projectivities $M\in\PGL_2(q)$ satisfy $MJ\ne JM$ and hence induce APN
permutations $F_M$ of $\Z_{q+1}$. Moreover, distinct projectivities $M$
induce distinct permutations for the fixed choices of $R$ and $P_0$.
\end{proposition}

\begin{proof}
Choose a matrix $A\in\operatorname{GL}_2(q)$ representing $J$.  Since $J$ has
order two in $\PGL_2(q)$, we have $A^2=\lambda I$ for some
$\lambda\in\F_q^*$.  This $\lambda$ is a nonsquare: otherwise the polynomial
$X^2-\lambda$, and hence the minimal polynomial of the nonscalar matrix $A$,
would split over $\F_q$, giving an
eigenline over $\F_q$, whereas a nonidentity element of the Singer group $C$
has no such eigenline.  Thus the minimal polynomial of $A$ is irreducible of
degree two and $\F_q[A]\simeq\F_{q^2}$.  Let
$B\in\operatorname{GL}_2(q)$.  Its projective class
centralizes $J$ if and only if $BAB^{-1}=\mu A$ for some
$\mu\in\F_q^*$.  Squaring this equality and using $A^2=\lambda I$ gives
$\mu^2=1$, so $\mu=1$ or $-1$.

If $BA=AB$, then $B$ is a nonzero element of the two-dimensional algebra
$\F_q[A]$; there are $q^2-1$ such matrices.  To count the solutions of
$BA=-AB$, identify the underlying vector space with $\F_{q^2}$ and let $A$
act as multiplication by an element $\beta$ satisfying $\beta^2=\lambda$.
The Frobenius automorphism $S:x\mapsto x^q$ is $\F_q$-linear and satisfies
$S A=-A S$, because $\beta^q=-\beta$.  If $U\in\F_q[A]^*$, then
$(SU)A=S(AU)=-A(SU)$, so $SU$ anticommutes with $A$. Conversely,
$S^2=1$ and $S^{-1}A=-AS^{-1}$; hence, if $BA=-AB$, then
$(S^{-1}B)A=-S^{-1}AB=A(S^{-1}B)$. Thus the anticommuting matrices are
exactly the matrices $SU$ with $U\in\F_q[A]^*$, and there are again $q^2-1$
of them. The commuting and
anticommuting sets are disjoint because $q$ is odd.  After quotienting by the
$q-1$ nonzero scalar matrices, this gives
\begin{equation*}
 |C_{\PGL_2(q)}(J)|=\frac{2(q^2-1)}{q-1}=2(q+1).
\end{equation*}

It remains to identify this centralizer with the normalizer of $C$.  Since
$J$ is the unique involution of $C$, every element of
$N_{\PGL_2(q)}(C)$ fixes $J$ under conjugation.  Conversely, the equations
$BA=AB$ and $BA=-AB$ show that a projectivity centralizing $J$ preserves or
interchanges its two conjugate eigenlines.  This unordered eigenpair
determines the nonsplit torus $C$, so the projectivity normalizes $C$.
Therefore $C_{\PGL_2(q)}(J)=N_{\PGL_2(q)}(C)$, in agreement with the standard
nonsplit-torus normalizer theorem~\cite{HirschfeldOUP1998}.  Since
$|\PGL_2(q)|=q(q^2-1)$, subtracting the $2(q+1)$ elements that centralize
$J$ gives \eqref{E:singer-count}.  Since $\exp_{R,P_0}$ is a bijection, $F_M=F_L$
implies that
$M$ and $L$ agree on every point $\exp_{R,P_0}(i)$, hence on the entire
projective line.  Two projectivities that agree on every point of the
projective line are equal, so $M=L$.  Therefore $M\mapsto F_M$ is
injective. This completes the proof.
\end{proof}

%
%
%

\begin{corollary}\label{C:singer-composition}
For $M,L\in\PGL_2(q)$, we have $F_M\circ F_L=F_{ML}$.  The composition is APN
if and only if $(ML)J\ne J(ML)$.
\end{corollary}

\begin{proof}
For every $i\in\Z_n$, the definition gives
\begin{equation*}
 F_M(F_L(i))
 =\log_{R,P_0}\bigl(M(\exp_{R,P_0}(F_L(i)))\bigr)
 =\log_{R,P_0}(ML(R^iP_0))=F_{ML}(i).
\end{equation*}
Here the middle equality uses
$\exp_{R,P_0}(F_L(i))=L(\exp_{R,P_0}(i))=L(R^iP_0)$, which follows directly
from \eqref{E:singer-induced}.
If $ML$ does not
centralize $J$, \cref{T:singer-apn} applies.  Otherwise,
$MLJ=JML$, and therefore
\begin{align*}
 \exp_{R,P_0}\bigl(F_{ML}(x+n/2)\bigr)
 &=ML(JR^xP_0)
 &=JML(R^xP_0)
  =\exp_{R,P_0}\bigl(F_{ML}(x)+n/2\bigr).
\end{align*}
The injectivity of $\exp_{R,P_0}$ gives
$F_{ML}(x+n/2)=F_{ML}(x)+n/2$ for every $x\in\Z_n$, so one DDT entry equals
$n$; explicitly, $\DDT_{F_{ML}}(n/2,n/2)=n$.  The permutation is therefore
not APN.
\end{proof}

For small prime values of $q$, we fixed a Singer cycle $R$ and a base point
$P_0$, and exhaustively enumerated every $M\in\PGL_2(q)$ satisfying $MJ\ne JM$.
Table~\ref{Tab:singer-ddt} gives the resulting differential spectra.  In the last
column, $m:(A_0,A_1,A_2)$ means that $m$ permutations have the indicated
spectrum.  Each row uses a fixed Singer cycle $R$ and a fixed base point
$P_0$, and the permutation counts
in its last column sum to the number in \eqref{E:singer-count}.

\begin{table}[htbp]
\caption{Differential spectra of Singer-cycle permutations at small orders}
\label{Tab:singer-ddt}
\centering
\footnotesize
\begin{tabularx}{\textwidth}{|c|c|c|c|>{\centering\arraybackslash}X|}
\hline
$q$&$n$&No. of $M$&No. of spectra&Permutations: differential spectrum\\ \hline
$3$ &$4$ &$16$    &$1$ &$16:(4,4,4)$\\ \hline
$5$ &$6$ &$108$   &$2$ &$72:(12,6,12);\ 36:(13,4,13)$\\ \hline
$7$ &$8$ &$320$   &$2$ &$192:(24,8,24);\ 128:(25,6,25)$\\ \hline
$11$&$12$&$1296$  &$2$ &$720:(60,12,60);\ 576:(61,10,61)$\\ \hline
$13$&$14$&$2156$  &$2$ &$1176:(84,14,84);\ 980:(85,12,85)$\\ \hline
$17$&$18$&$4860$  &$2$ &$2592:(144,18,144);\ 2268:(145,16,145)$\\ \hline
$19$&$20$&$6800$  &$2$ &$3600:(180,20,180);\ 3200:(181,18,181)$\\ \hline
$23$&$24$&$12096$ &$2$ &$6336:(264,24,264);\ 5760:(265,22,265)$\\ \hline
\end{tabularx}
\end{table}

Infinite subfamilies can also be obtained by restricting $q$ to arithmetic
progressions.  Dirichlet's theorem gives infinitely many primes
$q\equiv2\pmod9$, for which
$v_3(q+1)=1$, and infinitely many primes $q\equiv3\pmod8$, for which
$v_2(q+1)=2$.  These congruences describe the orders directly and do
not enter the APN proof.

For comparison with the order forms of the earlier finite field classes, an
order $n$ can have the form $r-j$, $1\le j\le4$, only if at least one of
$n+1,n+2,n+3,n+4$ is a prime power.  The next theorem excludes all four
possibilities.

\begin{theorem}\label{T:singer-separated}
The progression
\begin{equation}
 q\equiv269363\pmod{471240}                           \label{E:singer-progression}
\end{equation}
contains infinitely many primes.  For every such prime $q$, put $n=q+1$.
Then $v_3(n)=1$, $v_2(n)=2$, none of
$n+1,n+2,n+3,n+4$ is a prime power, and $n\ne2p$ for every prime $p$.
Consequently, the corresponding APN permutations in \cref{T:singer-apn} have orders
outside the Welch--Costas and Panario--Sakzad--Stevens--Wang order form $r-1$
and the standard Golomb order forms $r-2,r-3,r-4$, where $r$ is a prime power.
\end{theorem}

\begin{proof}
Reducing the congruence in \eqref{E:singer-progression} modulo the relevant
divisors of $471240$ gives $q\equiv2\pmod9$, $q\equiv3\pmod8$,
$q\equiv3\pmod5$, $q\equiv15\pmod{17}$, $q\equiv3\pmod7$, and
$q\equiv6\pmod{11}$. Since $\gcd(269363,471240)=1$, Dirichlet's theorem
guarantees infinitely many primes in this arithmetic progression.
For any such prime $q$, put $n=q+1$. The congruences modulo $9$ and $8$ give
$n\equiv3\pmod9$ and $n\equiv4\pmod8$. Hence $3$ divides $n$ but $9$ does
not, while $4$ divides $n$ but $8$ does not. The same residue information
gives $n+1=q+2\equiv0\pmod{85}$, $n+2=q+3\equiv6\pmod8$,
$n+3=q+4\equiv0\pmod{21}$, and $n+4=q+5\equiv0\pmod{22}$.
Since $85=5\cdot17$, $21=3\cdot7$, and $22=2\cdot11$, each of
$n+1$, $n+3$, and $n+4$ has at least two distinct prime divisors. Moreover,
$n+2\equiv6\pmod8$, so $n+2$ is divisible by $2$ but not by $4$. Since
$n+2>2$, its odd cofactor is greater than one and therefore has an odd prime
divisor. Thus none of $n+1,n+2,n+3,n+4$ is a prime power.
Finally, $12$ divides $n$, so $6$ divides $n/2$. Since $n/2>6$, the integer
$n/2$ is composite, and hence $n$ cannot be written as $2p$ with $p$ prime.
\end{proof}

\cref{Tab:singer-orders} gives smaller instances to show the same
comparison with the order forms of the earlier classes.  In every row, each
of $n+1,n+2,n+3,n+4$ has at least two distinct prime divisors and therefore is
not a prime power.

\begin{table}[htbp]
\caption{Singer orders not of the forms $r-1,r-2,r-3,r-4$}
\label{Tab:singer-orders}
\centering
\footnotesize
\begin{tabular}{|c|c|c|}
\hline
$q$ & $n=q+1$ & Factorizations of $n+1,n+2,n+3,n+4$\\
\hline
$83$  & $84$  & $5\cdot17,\ 2\cdot43,\ 3\cdot29,\ 2^3\cdot11$\\ \hline
$131$ & $132$ & $7\cdot19,\ 2\cdot67,\ 3^3\cdot5,\ 2^3\cdot17$\\ \hline
$173$ & $174$ & $5^2\cdot7,\ 2^4\cdot11,\ 3\cdot59,\ 2\cdot89$\\ \hline
$257$ & $258$ & $7\cdot37,\ 2^2\cdot5\cdot13,\ 3^2\cdot29,\ 2\cdot131$\\ \hline
$443$ & $444$ & $5\cdot89,\ 2\cdot223,\ 3\cdot149,\ 2^6\cdot7$\\ \hline
\end{tabular}
\end{table}

\begin{example}\label{Ex:z258}
For the order near $256$ in \cref{Tab:singer-orders}, take $q=257$,
$f(X)=X^2-6X+3$, and
$K=\F_{257}[X]/(f(X))$.  Let $\alpha$ be the residue class of $X$.
The discriminant $24$ is a nonsquare in $\F_{257}$, and
$ \alpha^{33024}=-1,
 \alpha^{22016}=179+240\alpha,
 \alpha^{1536}=125+176\alpha
$.
Since $66048=257^2-1$ has prime divisors $2,3,43$, these identities show that
$\alpha$ is primitive.  With respect to $(1,\alpha)$, let us take $P_0=\infty=[1:0]$,
\begin{equation*}
 R=\begin{pmatrix}0&254\\1&6\end{pmatrix}~~\text{and}~~
 M=\begin{pmatrix}1&1\\0&1\end{pmatrix}
\end{equation*}
over $\F_{257}$.  Thus $R$ represents multiplication by $\alpha$ and
$P_i=[R^i(1,0)^{\mathsf T}]=[\alpha^i]$.  For every $i$, the value $F_M(i)$
is the unique $j$ for which
$MR^i(1,0)^{\mathsf T}\in\F_{257}^*R^j(1,0)^{\mathsf T}$.
The matrix $R$ induces a $258$-cycle on the projective line,
and $$J=R^{129}=\left(\begin{matrix}90&90\\227&167\end{matrix}\right)$$
does not commute with $M$.  Hence $F_M$ is an APN permutation of $\Z_{258}$.
Its truth table and differential spectrum are given in
Section~\ref{A:TruthTables}.
\end{example}

\section{APN Permutations Induced by Binomials}\label{S:Binomial}

The second construction labels the points of $\PG(d-1,q)$, identified with
$\F_{q^d}^*/\F_q^*$, by the exponents of a primitive element.  A binomial whose
two exponents are congruent modulo
$q-1$ induces a function on these points.  If the kernel in $\F_{q^d}^*$ of the
power function defined by the exponent difference is exactly $\F_q^*$, then each
scalar in $\F_q^*$ contributes at most one projective point satisfying a
differential equation.  We prove this statement for every binomial that
induces a permutation of the projective points; no additivity or
linearized-polynomial assumption is used.  An explicit family over
$\F_{3^d}$ is then obtained from invertible linearized polynomials.

\subsection{A general condition on the exponents}

\begin{construction}\label{Con:binomial}
Let $q>2$ be a prime power, let $d\ge2$, let $K=\F_{q^d}$, and choose a
primitive element $\alpha\in K^*$.  Put $n_q=(q^d-1)/(q-1)$.  For
nonnegative integers $r,t$ satisfying $r\equiv t\pmod{q-1}$ and
$u,v\in K^*$, put $H(x)=ux^r+vx^t$.  Suppose that $H(x)\ne0$ for
$x\ne0$ and that
$[x]\mapsto[H(x)]$ is a permutation of the points of $\PG(d-1,q)$, identified
with $K^*/\F_q^*$.  For each $i\in\Z_{n_q}$, define $F_H(i)$ to be the
unique $j\in\Z_{n_q}$ such that $H(\alpha^i)\in\F_q^*\alpha^j$.
Then $F_H$ is a permutation of $\Z_{n_q}$.
\end{construction}

The congruence between $r$ and $t$ makes this definition independent of the
representative of $[\alpha^i]$.  Exponents on $K^*$ are understood modulo
$q^d-1$; interchanging the two terms permits us to take $r\ge t$ when a
nonnegative exponent difference is convenient.  The condition $H(x)\ne0$
ensures that the projective function $[x]\mapsto[H(x)]$ is defined everywhere;
this requirement is implicit in the assumption that the induced function is a permutation.

\begin{theorem}\label{T:exponent-condition}
In \cref{Con:binomial}, assume that
\begin{equation}
 \gcd(r-t,q^d-1)=q-1.                              \label{E:exponent-condition}
\end{equation}
Then $\du(F_H)\le q-1$.  In particular, when $q=3$, every binomial satisfying
the hypotheses of \cref{Con:binomial} and \eqref{E:exponent-condition}
induces an APN permutation of $\Z_{(3^d-1)/2}$.
\end{theorem}

\begin{proof}
 According to the definition of differential uniformity in \cref{D:APN}, we
 need to prove that the differential equation
 $F_H(i+a)-F_H(i)=b$ has at most $q-1$ solutions $i\in\Z_{n_q}$ for every
 $a\in\Z_{n_q}\setminus\{0\}$ and $b\in\Z_{n_q}$; when $q=3$, this bound is
 two.  Fix such $a$ and $b$, and put $A=\alpha^a$ and
 $B=\alpha^b$.  Both $A$ and $B$ are nonzero.  Write $x=\alpha^i$ and
 $j=F_H(i)$.  Then $H(x)=\mu\alpha^j$ for some $\mu\in\F_q^*$. By the
definition of $F_H$, the equality $F_H(i+a)=j+b$ holds if and only if
$[H(Ax)]=[\alpha^{j+b}]$. Since $B=\alpha^b$ and
$[H(x)]=[\alpha^j]$, we have $[\alpha^{j+b}]=[BH(x)]$. Thus the residue $i$
solves $F_H(i+a)-F_H(i)=b$ if and only if
$H(Ax)=\lambda BH(x)$ for some $\lambda\in\F_q^*$. After substituting the
binomial $H$, this equality becomes
\begin{equation}
 u(A^r-\lambda B)x^r+v(A^t-\lambda B)x^t=0.       \label{E:binomial-difference}
\end{equation}
The condition \eqref{E:exponent-condition} implies $r\ne t$; after
interchanging the two terms if necessary, assume $r>t$.
The two coefficients in \eqref{E:binomial-difference} cannot vanish
simultaneously.  Indeed, simultaneous vanishing would give
$A^{r-t}=1$.  Since $r\equiv t\pmod{q-1}$, every element of $\F_q^*$ lies in
the kernel of $z\mapsto z^{r-t}$.  By \eqref{E:exponent-condition}, this
kernel has order $q-1$ and is therefore exactly $\F_q^*$.  Hence
$A=\alpha^a\in\F_q^*$, which is equivalent to $n_q\mid a$, or to $a=0$ in
$\Z_{n_q}$.

 We first consider the general case in which both coefficients in
 \eqref{E:binomial-difference} are nonzero.  Division by $x^t$ reduces the
 equation to
 $x^{r-t}=C_\lambda$, where
\begin{equation*}
 C_\lambda=-\frac{v(A^t-\lambda B)}{u(A^r-\lambda B)}\ne0.
\end{equation*}
If $x_1$ and $x_2$
are two roots, then $(x_1/x_2)^{r-t}=1$, so
 $x_1/x_2\in\F_q^*$ by the kernel identity above.  Thus all roots represent
 one projective point in $K^*/\F_q^*$.

 We next consider the special case in which one coefficient in
 \eqref{E:binomial-difference} vanishes.  Since the two coefficients cannot
 vanish simultaneously, the equation has no nonzero root in this case.

 Combining the two cases, each fixed $\lambda\in\F_q^*$ gives at most one
 projective solution.  Conversely, for a fixed projective point $[x]$, the
 relation $H(Ax)=\lambda BH(x)$ determines at most one $\lambda$, because
$BH(x)\ne0$.  This value does not depend on the representative: if
$\gamma\in\F_q^*$, then $\gamma^r=\gamma^t$ and both sides of the relation
for $\gamma x$ acquire the same factor $\gamma^t$.  There are $q-1$ possible
values of $\lambda$, so the differential equation is satisfied by at most
$q-1$ projective points.  Since
 $i\mapsto[\alpha^i]$ is a bijection from $\Z_{n_q}$ to
 $K^*/\F_q^*$, these projective points give at most $q-1$ solutions for $i$.
 This proves $\du(F_H)\le q-1$.  For $q=3$, the upper bound is two.  The
 general lower bound for permutations gives $\du(F_H)=2$, so $F_H$ is APN.
\end{proof}

The hypothesis that $[x]\mapsto[H(x)]$ is a permutation is independent of the
greatest-common-divisor condition in \eqref{E:exponent-condition}.  Therefore
\cref{T:exponent-condition} applies whenever the first hypothesis is proved,
whether or not $H$ arises from a linearized polynomial.

\subsection{An explicit family over \texorpdfstring{$\F_{3^d}$}{F3d}}

Let $d\ge2$, $s\ge1$ satisfy $\gcd(s,d)=1$, and put
$n=(3^d-1)/2$.  Choose $c\in\F_{3^d}^*$ such that $-c$ is a nonsquare, and
choose a positive integer $k$ with $\gcd(k,n)=1$.  Define
$H_{k,c,s}(x)=x^{k3^s}+c x^k$.

\begin{theorem}\label{T:ternary-apn}
Let $\alpha$ be primitive in $\F_{3^d}$.  For each $i\in\Z_n$, define
$F_{k,c,s}(i)$ to be the unique $j\in\Z_n$ such that
$H_{k,c,s}(\alpha^i)\in\{\alpha^j,-\alpha^j\}$.
Then $F_{k,c,s}$ is an APN permutation.  In particular, $c=-\alpha$ gives an
explicit member for every $d\ge2$, every $s$ coprime to $d$, and every unit
$k$ modulo $n$.
\end{theorem}

\begin{proof}
The quotient $\F_{3^d}^*/\{\pm1\}$ is cyclic of order $n$.  Hence the power
function $[x]\mapsto[x^k]$ is bijective on this quotient because
$\gcd(k,n)=1$; it is well defined since $x'=\pm x$ implies
$(x')^k=\pm x^k$.  The transformation
$L_{c,s}(y)=y^{3^s}+cy$ has a nonzero kernel element precisely when
$y^{3^s-1}=-c$.  The identity
\begin{equation}
 \gcd(3^s-1,3^d-1)=2                              \label{E:gcd-linearized}
\end{equation}
follows from the standard identity
$\gcd(3^s-1,3^d-1)=3^{\gcd(s,d)}-1$ together with $\gcd(s,d)=1$.  Since
$\F_{3^d}^*$ is cyclic of order $3^d-1$, the kernel of
$y\mapsto y^{3^s-1}$ has order two.  Its image is therefore a subgroup of
$\F_{3^d}^*$ of order $(3^d-1)/2$. Since the cyclic group
$\F_{3^d}^*$ has a unique subgroup of this order, the image is precisely the
subgroup of squares.  Hence the nonsquare $-c$ does not belong to the image,
and $L_{c,s}$ has trivial kernel.  As $L_{c,s}$ is an $\F_3$-linear
transformation on the finite-dimensional vector space $\F_{3^d}$, it is invertible.
It also respects multiplication by $\pm1$, because
$L_{c,s}(-y)=-L_{c,s}(y)$.  Thus $L_{c,s}$ induces a permutation of
$\F_{3^d}^*/\{\pm1\}$.  Since
$H_{k,c,s}(x)=L_{c,s}(x^k)$, its induced projective permutation is the composition of
this permutation with the power permutation above, and is therefore a
permutation.  Moreover,
writing $3^s-1=2h$, \eqref{E:gcd-linearized} gives $\gcd(h,n)=1$, and thus
$ \gcd\bigl(k(3^s-1),3^d-1\bigr)=2$.
Indeed, $3^d-1=2n$ and the last greatest common divisor equals
$2\gcd(kh,n)=2$.  Finally, $3^d-1$ is even, and a generator of the cyclic
group $\F_{3^d}^*$ cannot lie in its subgroup of squares.  Thus the primitive
element $\alpha$ is a nonsquare, and $c=-\alpha$ always satisfies the required
condition because $-c=\alpha$.
The result now follows from \cref{T:exponent-condition} with $q=3$.
\end{proof}

The identity $F_{k,c,s}=F_{1,c,s}\circ\mu_k$, where
$\mu_k(i)=ki$ on $\Z_n$, shows that changing $k$ only precomposes $F_{1,c,s}$ with
a group automorphism.  Therefore it is not claimed to yield a new equivalence
class.  Since $\mu_k$ is a group automorphism,
precomposition by $\mu_k$ preserves every DDT multiplicity up to a
permutation of the input differences.  The hypothesis of
\cref{T:exponent-condition} is not restricted to binomials obtained from the
linearized polynomials used in this subsection.

\subsection{Examples and orders}

\begin{example}\label{Ex:ternary-small}
Let $f(X)=X^3+2X^2+X+1$. Since $f$ is a primitive irreducible polynomial
over $\F_3$, the quotient $K=\F_3[X]/(f(X))$ is a field with $27$ elements.
Let $\alpha=X+(f(X))$ be the residue class of $X$ in $K$. Then
$f(\alpha)=0$, and $\alpha$ is a primitive element of $K$; equivalently,
$K^*=\langle\alpha\rangle$ and $\alpha$ has multiplicative order $26$.
Take $d=3$, $s=1$, $c=-\alpha$, and $k=5$, so that the domain has order
$n=(3^3-1)/2=13$.  Here
$H(x)=x^{15}-\alpha x^5=L(x^5)$ is not a linearized polynomial, although its
permutation property is supplied by the linearized transformation
$L(x)=x^3-\alpha x$.  In the ordered power basis $(1,\alpha,\alpha^2)$,
multiplication by $\alpha$ and $L$ have the matrices
\begin{equation*}
 R=\begin{pmatrix}0&0&2\\1&0&2\\0&1&1\end{pmatrix}~~\text{and}~~
 L=\begin{pmatrix}1&2&0\\2&2&0\\0&0&2\end{pmatrix}.
\end{equation*}
Put $e_0=(1,0,0)^{\mathsf T}$.  The vectors $R^ie_0$, up to multiplication
by $\pm1$, enumerate the $13$ one-dimensional subspaces.  The value $F(i)$
is the unique $j$ satisfying
$LR^{5i}e_0\in\{R^je_0,-R^je_0\}$.
Applying $H$ to the $13$ projective points and expressing the images in
terms of their cyclic labels gives the truth table
\[
[F(0),\ldots,F(12)]
 =[4,3,12,2,8,9,0,11,6,1,10,7,5].
\]
This permutation has differential uniformity two, and its differential spectrum is
represented by
$(A_0,A_1,A_2)=(36,84,36)$.  The difference between the two exponents is $10$, and
$\gcd(10,26)=2$, so its APN property follows directly from
\cref{T:exponent-condition}.  The identity
$F_{5,-\alpha,1}=F_{1,-\alpha,1}\circ\mu_5$ also makes the relation to the
representative with $k=1$ explicit.
\end{example}

Table~\ref{Tab:linearized-ddt} exhausts the pairs $(s,c)$ with
$1\le s<d$, $\gcd(s,d)=1$, and $-c$ nonsquare for each value of $d$
listed in the table, with a fixed primitive element $\alpha$ and $k=1$.  Direct comparison of their
truth tables shows computationally that no two parameter pairs in these rows
induce the same permutation.  The permutations obtained by varying $k$ have the same
differential spectra and are not counted again.  For the parameter ranges listed in the table, all
permutations with $k=1$ at a fixed order have the same spectrum; this is a
computational observation, not an additional general theorem.

\begin{table}[htbp]
\caption{Differential spectra of the ternary-binomial permutations with $k=1$}
\label{Tab:linearized-ddt}
\centering
\footnotesize
\begin{tabularx}{\textwidth}{|c|c|c|c|c|>{\centering\arraybackslash}X|}
\hline
$d$&$n$&Pairs $(s,c)$&Distinct permutations&No. of spectra&Permutations: differential spectrum\\ \hline
$2$&$4$   &$4$   &$4$   &$1$ &$4:(4,4,4)$\\ \hline
$3$&$13$  &$26$  &$26$  &$1$ &$26:(36,84,36)$\\ \hline
$4$&$40$  &$80$  &$80$  &$1$ &$80:(400,760,400)$\\ \hline
$5$&$121$ &$484$ &$484$ &$1$ &$484:(3600,7320,3600)$\\ \hline
\end{tabularx}
\end{table}

For an order $n$ in this section, the order forms $r-1,r-2,r-3,r-4$ of the
earlier finite field classes require $n+1,n+2,n+3,n+4$, respectively, to be
prime powers.  An order $q+1$ from the Singer construction requires $n-1$ to
be a prime power.  The next theorem excludes all five possibilities
simultaneously.

\begin{theorem}\label{T:ternary-separated}
Let $d\ge14$ be an integer satisfying $d\equiv14\pmod{4620}$, and put $n=(3^d-1)/2$.  Then the APN permutations
in \cref{T:ternary-apn} have orders for which none of
\begin{equation}
 n-1,n+1,n+2,n+3,n+4                                  \label{E:ternary-neighbors}
\end{equation}
is a prime power.  Moreover, $n\ne2p$ for every prime $p$.
\end{theorem}

\begin{proof}
The modulus $4620$ is divisible by $4$, $20$, $42$, and $110$. Direct
modular exponentiation gives $3^{14}\equiv19\pmod{25}$,
$3^{14}\equiv30\pmod{49}$, and $3^{14}\equiv81\pmod{121}$. The
multiplicative orders of $3$ modulo $25$, $49$, and $121$ divide $20$, $42$,
and $110$, respectively. Since $d\equiv14\pmod{4620}$ and each of
$20$, $42$, and $110$ divides $4620$, the same three congruences hold with
$14$ replaced by $d$. Therefore,
$n+1=(3^d+1)/2\equiv10\pmod{25}$,
$n+3=(3^d+5)/2\equiv42\pmod{49}$, and
$n+4=(3^d+7)/2\equiv44\pmod{121}$.
These congruences give
$v_5(n+1)=1$, $v_7(n+3)=1$, and $v_{11}(n+4)=1$.
Furthermore, we have $n-1=3(3^{d-1}-1)/2$ and
$n+2=3(3^{d-1}+1)/2$.
Since $d-1\ge1$, we have
$3^{d-1}-1\equiv-1\pmod3$ and $3^{d-1}+1\equiv1\pmod3$.  Thus the second
factor in each expression is greater than one and is not divisible by $3$, so
$v_3(n-1)=v_3(n+2)=1$.  Hence $n-1$ and $n+2$ have the
prime divisor $3$ with valuation one and are larger than $3$; similarly,
$n+1,n+3,n+4$ have the prime divisors $5,7,11$, respectively, with valuation
one and are larger than those primes.  None of the five integers is therefore
a prime power.  As $d$ ranges over the progression $14\pmod{4620}$, the value
$n=(3^d-1)/2$ increases strictly and hence gives infinitely many distinct
orders.  Finally, $d\equiv2\pmod4$, and the
lifting-the-exponent formula \eqref{E:lte} gives
$ v_2(3^d-1)=v_2(2)+v_2(4)+v_2(d)-1=3$,
since $v_2(d)=1$.  Hence $v_2(n)=2$.  If $n=2p$ with $p$ an odd prime,
then $v_2(n)=1$; if $p=2$, then $n=4$.  Since $n>4$, neither case is possible.
\end{proof}

\begin{example}\label{Ex:ternary-separated}
Let $K=\F_3[X]/(X^{14}+X+2)$ and let $\alpha$ be the residue class of $X$;
the polynomial $X^{14}+X+2$ is primitive. Take $d=14$, $s=1$, $c=-\alpha$,
and $k=1$. For each $i\in\Z_n$, define $F(i)$ as the unique $j\in\Z_n$
such that $\alpha^{3i}-\alpha^{i+1}\in\{\alpha^j,-\alpha^j\}$. The order is
$n=(3^{14}-1)/2=2{,}391{,}484$. The integers in
\eqref{E:ternary-neighbors} factor as $n-1=3\cdot797161$,
$n+1=5\cdot29\cdot16493$, $n+2=2\cdot3\cdot398581$,
$n+3=7\cdot341641$, and $n+4=2^6\cdot11\cdot43\cdot79$. The APN property
follows from \cref{T:ternary-apn}.
\end{example}

\section{APN Permutations from Completed Reciprocals}\label{S:Reciprocal}

The third construction is given directly as a function of a residue
$z\in\Z_{2p}$.  One formula involving a completed reciprocal is used when
$z$ is even, and another is used when $z$ is odd.  The output parity is
defined by the Legendre symbol.  In the proof, two roots of the same
quadratic equation have different output parities.  The construction
requires neither a discrete logarithm nor a labeling of projective points.

\subsection{The construction}

\begin{construction}\label{Con:reciprocal}
Let $p\ge11$ be a prime, $a\in\F_p^*$, and $c\in\F_p$.  For
$z\in\Z_{2p}$, let $E(z)=(1-(-1)^z)/2$ and $X_c(z)\in\F_p$ be the
residue class of $z+cE(z)$ modulo $p$.  In $\F_p$, set
$A_a(z)=1+(a-1)E(z)$ and
$Y_{a,c,p}(z)=A_a(z)X_c(z)^{p-2}$.  Let
$R_{c,p}(z)\in\{0,1\}$ be congruent to
$\rho(X_c(z))+E(z)(1-\chi(X_c(z))^2)$ modulo $2$. We define
\begin{equation}
 \mathcal B_{a,c,p}(z)
 \equiv pR_{c,p}(z)+(p+1)Y_{a,c,p}(z)\pmod{2p}.       \label{E:reciprocal-map}
\end{equation}
\end{construction}

The quantities $(-1)^z$ and $E(z)$ are well defined on $\Z_{2p}$, and
$E(z)$ is the parity of $z$.  Moreover,
$X_c(z)^{p-2}=\iota(X_c(z))$, including the value at zero.  Let
$e=z\bmod2$ and $x=z\bmod p$.  If $e=0$, then $X_c(z)=x$,
$Y_{a,c,p}(z)=\iota(x)$, and $R_{c,p}(z)=\rho(x)$.  If $e=1$, then
$X_c(z)=x+c$, $Y_{a,c,p}(z)=a\iota(x+c)$, and $R_{c,p}(z)$ is congruent to
$\rho(x+c)+1-\chi(x+c)^2$ modulo $2$.
The denominator is zero at $x=0$ when $e=0$ and at $x=-c$ when $e=1$; the
completed reciprocal assigns the value zero at both inputs.  The variables
$e$ and $x$ are not independent because both are determined by $z$.  When
$\chi(a)=-1$, the output parity is equivalently
$e+\rho(Y_{a,c,p}(z))\pmod2$, also when $X_c(z)=0$.

\begin{proposition}\label{P:reciprocal-permutation}
If $\chi(a)=-1$, then the function $\mathcal B_{a,c,p}$ in
\eqref{E:reciprocal-map} is a
permutation of $\Z_{2p}$.
\end{proposition}

\begin{proof}
Write $e=E(z)$ and $x=z\bmod p$.  The output residue is
\begin{equation*}
 Y_{a,c,p}(z)=
 \begin{cases}
  \iota(x),&e=0,\\
  a\iota(x+c),&e=1.
 \end{cases}
\end{equation*}
 In both cases, the specified function of $x$ is a permutation of $\F_p$.  If the output residue is
$y\ne0$, inversion preserves square class and multiplication by the
nonsquare $a$ reverses it.  The definition of $R_{c,p}$ therefore gives output parity
$e+\rho(y)$ in both cases.  At $y=0$, the input satisfying $e=0$ and $x=0$
has output coordinates $(0,0)$, whereas the input satisfying $e=1$ and
$x=-c$ has output coordinates $(1,0)$.  The same identity therefore holds
when the reciprocal argument is zero.  For a fixed $y$, the equation
$\iota(x)=y$ has exactly one solution with $e=0$, and
$a\iota(x+c)=y$ has exactly one solution with $e=1$.  These solutions
produce the two output parities $\rho(y)$ and $1+\rho(y)$, respectively.
By the Chinese remainder theorem, a residue modulo $2p$ is uniquely
determined by its parity and its residue modulo $p$.  Hence every residue
modulo $2p$ is attained exactly once.
\end{proof}

For the differential analysis, write
$\phi_e(x)=a^e\iota(x+ce)$ for $e\in\{0,1\}$.  The residue modulo $p$ of the
output is $\phi_e(x)$, and its parity is $e+\rho(\phi_e(x))$.  The following
lemma gives an explicit relation between two roots arising from an odd input
difference.

\begin{lemma}\label{L:opposite-character-roots}
Assume $\chi(a)=-1$ and put
$R_a(z)=((a-1)z-1)/(z(z+1))$.  If $R_a(z)=u$ has two distinct roots
$z,z'\in\F_p\setminus\{0,-1\}$, then
\begin{equation}
 z'=\frac{z+1}{(a-1)z-1}                              \label{E:second-root}
\end{equation}
and
\begin{equation}
 \chi\bigl(z'(z'+1)\bigr)=-\chi\bigl(z(z+1)\bigr). \label{E:opposite-characters}
\end{equation}
\end{lemma}

\begin{proof}
The equation $R_a(z)=u$ is $uz^2+(u-a+1)z+1=0$.  If $u=0$, this equation is
linear because $a\ne1$, and hence it cannot have two distinct roots.  Thus
$u\ne0$, and the product and sum of the roots give $zz'=1/u$.  Moreover,
$((a-1)z-1)=uz(z+1)\ne0$.  Consequently,
\begin{equation*}
 z'=\frac{1}{uz}=\frac{z+1}{(a-1)z-1},
\end{equation*}
which proves \eqref{E:second-root}.  Direct substitution gives
\begin{equation*}
 z'(z'+1)=\frac{a\,z(z+1)}{((a-1)z-1)^2}.
\end{equation*}
Taking quadratic characters and using $\chi(a)=-1$ proves
\eqref{E:opposite-characters}.
\end{proof}

\subsection{The APN property}

Every input difference in $\Z_{2p}$ is uniquely described by its parity
$\varepsilon\in\F_2$ and its residue $h\in\F_p$.  If an input has coordinates
$(e,x)$, then adding this difference gives $(e+\varepsilon,x+h)$, where the
first coordinate is reduced modulo $2$.  Similarly, an output difference has
coordinates $(\eta,b)\in\F_2\times\F_p$.  Consequently, a DDT entry is fixed
by specifying both $\eta$ and $b$.  We first consider input differences with
$\varepsilon=0$ and then those with $\varepsilon=1$.

\begin{lemma}\label{L:reciprocal-even}
Assume that $ \chi(a)=\chi(-3)=\chi(1-4a)=\chi(1-4/a)=-1$.
For every nonzero input difference with parity component zero, every DDT entry
of $\mathcal B_{a,c,p}$ is at most two.
\end{lemma}

\begin{proof}
 Let $d\in\Z_{2p}$ and $v\in\Z_{2p}$ have coordinates $(0,h)$ and
 $(\eta,b)$, respectively, where $h\ne0$.  According to the definition of
 differential uniformity in \cref{D:APN}, we need to prove that the differential equation
 $\mathcal B_{a,c,p}(z+d)-\mathcal B_{a,c,p}(z)=v$ has at most two solutions
 $z\in\Z_{2p}$ for every $h\in\F_p^*$ and $(\eta,b)\in\F_2\times\F_p$.
 Writing the input as $(e,x)$, this equation is equivalent to
 \begin{equation*}
  \begin{cases}
   \phi_e(x+h)-\phi_e(x)=b,\\
   \rho\bigl(\phi_e(x+h)\bigr)+\rho\bigl(\phi_e(x)\bigr)
      \equiv\eta\pmod2.
  \end{cases}
 \end{equation*}
 We first consider the general case in which both reciprocal arguments are
 nonzero, and then the special case in which one of them is zero.

 For an input of parity $e$, put $X=x+ec$.  If $X\ne0,-h$, the output
 residue difference is
 \begin{equation*}
 b=a^e\iota(X+h)-a^e\iota(X)
   =-\frac{a^eh}{X(X+h)}.
\end{equation*}
After the normalization $X=hz$, this identity becomes
\begin{equation}
 bh=-\frac{a^e}{z(z+1)}.                              \label{E:even-normalized}
\end{equation}
Equivalently, we have
$bhz^2+bhz+a^e=0$.  If $b=0$, this equation has no solution; if $b\ne0$, it
is quadratic.  Thus, for each fixed $e$, at most two solutions have both
$X\ne0$ and $X+h\ne0$.
For nonzero $r_0,r_1\in\F_p$, the definition of $\rho$ gives
\begin{equation*}
 \rho(r_1)+\rho(r_0)
 \equiv\frac{1-\chi(r_0r_1)}{2}\pmod2,
\end{equation*}
because the left-hand side is one precisely when $r_0$ and $r_1$ have
opposite quadratic characters.  Applying this identity to the two output
residues gives the parity component $\eta$
\begin{equation*}
 \begin{split}
 \eta&=\rho\bigl(a^e\iota(X+h)\bigr)
      +\rho\bigl(a^e\iota(X)\bigr)\pmod2\\
 &=\frac{1-\chi\bigl(a^{2e}/(X(X+h))\bigr)}{2}
  =\frac{1-\chi(X(X+h))}{2}
  =\frac{1-\chi(z(z+1))}{2}.
 \end{split}
\end{equation*}
Here $a^{2e}$ and $h^2$ are squares.  For the same $b$ and $h$,
\eqref{E:even-normalized} gives
$z_e(z_e+1)=-a^e/(bh)$.  Since $a$ is a nonsquare, the equations for $e=0$
and $e=1$ give opposite values of $\eta$.  Therefore, after $(\eta,b)$ is fixed, solutions
with nonzero denominators can occur for at most one value of $e$, and there
are at most two such solutions.

We now consider the exceptional values $X=0$ and $X=-h$, for which one of
the reciprocal arguments is zero.
For a fixed $e$, the two values $X=0,-h$ both give the output residue
 difference $b=a^e/h$ because the completed reciprocal equals zero at the
 zero argument.  Their parity components are $\eta_0=\rho(a^e/h)$ and
 $\eta_{-h}=\rho(-a^e/h)$, respectively. They are equal when
 $\chi(-1)=1$ and different when
 $\chi(-1)=-1$. Consequently, for a fixed pair $(\eta,b)$, either one or
both of the values $X=0,-h$ may contribute, and their total contribution is
at most two. Since $a\ne1$, the value $b=a^e/h$ cannot arise from a zero
denominator for inputs of parity $1-e$.
For this value of $b$, \eqref{E:even-normalized} gives
$z(z+1)=-1$ for inputs of parity $e$. For inputs of parity $1-e$, it gives
$z(z+1)=-a$ when $e=0$ and $z(z+1)=-a^{-1}$ when $e=1$. Thus a
nonzero-denominator solution of parity $e$ would satisfy
$z^2+z+1=0$, whereas one of parity $1-e$ would satisfy
$z^2+z+a=0$ when $e=0$ and $z^2+z+a^{-1}=0$ when $e=1$.
 The corresponding discriminants are $-3$, $1-4a$, and $1-4/a$,
 respectively. They are all nonsquares by hypothesis, so none of these
 equations has a solution in $\F_p$.

 Combining the nonzero- and zero-denominator cases, we conclude that a fixed
pair $(\eta,b)$ receives at most two solutions. If $b$ is not of the form
$a^e/h$, all contributing solutions have nonzero denominators, and the
preceding quadratic argument gives at most two. If $b=a^e/h$ for some $e$,
only the inputs with $X=0$ or $X=-h$ can contribute, and there are at most
two of them.
\end{proof}

\begin{lemma}\label{L:reciprocal-odd}
 Assume that $\chi(a)=\chi(1-4a)=\chi(1-4/a)=-1$.  For every input difference
 with parity component one, every DDT entry of $\mathcal B_{a,c,p}$ is at most
 two.
\end{lemma}

\begin{proof}
 Let $d\in\Z_{2p}$ and $v\in\Z_{2p}$ have coordinates $(1,h)$ and
 $(\eta,b)$, respectively.  According to the definition of differential
 uniformity in \cref{D:APN}, we need to prove that the differential equation
 $\mathcal B_{a,c,p}(z+d)-\mathcal B_{a,c,p}(z)=v$ has at most two solutions
 $z\in\Z_{2p}$ for every $h\in\F_p$ and
 $(\eta,b)\in\F_2\times\F_p$.  Writing the input as $(e,x)$, we obtain the
 following two systems:
 \begin{equation*}
  \begin{aligned}
   e=0:&\quad
   \begin{cases}
    a\iota(x+h+c)-\iota(x)=b,\\
    1+\rho\bigl(a\iota(x+h+c)\bigr)+\rho\bigl(\iota(x)\bigr)
       \equiv\eta\pmod2;
   \end{cases}\\[2mm]
   e=1:&\quad
   \begin{cases}
    \iota(x+h)-a\iota(x+c)=b,\\
    1+\rho\bigl(\iota(x+h)\bigr)+\rho\bigl(a\iota(x+c)\bigr)
       \equiv\eta\pmod2.
   \end{cases}
  \end{aligned}
 \end{equation*}
 Put $s=h+c$ and $t=h-c$.  We first consider the general case in which
 $s,t\ne0$ and both reciprocal arguments are nonzero.  We then consider zero
 reciprocal arguments and the degenerate cases $s=0$ or $t=0$.

 For an initial input of parity zero, its parity becomes one, the field
 coordinate changes from $X=x$ to $X+s$, and the output residue difference is
 $b=a\iota(X+s)-\iota(X)$.  If $s\ne0$, normalization by $X=sz$ gives
\begin{equation}
 bs=R_a(z)=\frac{(a-1)z-1}{z(z+1)}.                  \label{E:odd-forward}
\end{equation}
 For an initial input of parity one, its parity becomes zero,
the field coordinate changes from $X=x+c$ to $X+t$, and the output residue
difference is
$b=\iota(X+t)-a\iota(X)$.  If $t\ne0$, normalization by $X=tz$ gives
$bt=((1-a)z-a)/(z(z+1))=R_a(-z-1)$.
 We now determine which roots of these normalized equations can have a
prescribed output parity.
Whenever both reciprocal arguments are nonzero, the parity component of the
output difference is
\begin{equation*}
 \eta=\frac{1-\chi(z(z+1))}{2}\pmod2.
\end{equation*}
When the initial input has parity zero, the two nonzero output residues have product
$a/(X(X+s))$.  Since $X=sz$ and $\chi(a)=-1$,
\begin{equation*}
 \chi\left(\frac{a}{X(X+s)}\right)
 =-\chi(z(z+1)).
\end{equation*}
The change of input parity contributes one to the output parity difference.
Therefore
\begin{align*}
 \eta
 &\equiv1+\rho\bigl(a\iota(X+s)\bigr)+\rho\bigl(\iota(X)\bigr)\pmod2\\
 &\equiv1+\frac{1+\chi(z(z+1))}{2}
  \equiv\frac{1-\chi(z(z+1))}{2}\pmod2.
\end{align*}
When the initial input has parity one, the product of the two output residues
is $a/(X(X+t))$, and the same calculation with $X=tz$ gives the identical
formula.
By \cref{L:opposite-character-roots}, two roots of either normalized equation
have opposite values of $\eta$.  After denominators are cleared, each
normalized equation has degree at most two, and hence has at most two roots;
if it has two, only one can match the prescribed output parity.  Hence a
fixed pair $(\eta,b)$ is obtained from at most one root for each initial input
parity.

It remains to treat the exceptional values at which one of the reciprocal
arguments is zero. For an initial input of parity zero, the values $X=0$ and $X=-s$ give the
distinct normalized output residues $a$ and $1$, respectively. A
nonzero-denominator solution with normalized output residue $a$ would have
to satisfy $R_a(z)=a$, or equivalently $az^2+z+1=0$. Similarly, normalized
output residue $1$ would require $R_a(z)=1$, or equivalently
$z^2+(2-a)z+1=0$. The discriminants of these two equations are $1-4a$ and
$a(a-4)=a^2(1-4/a)$, respectively, and both are nonsquares by hypothesis.
For an initial input of parity one,
the same two equations are obtained after replacing $z$ by $-z-1$; the
values $X=0,-t$ give the normalized residues $1$ and $a$, respectively.
Thus, for either initial input parity and a fixed pair $(\eta,b)$, at most one
solution can occur: it has either two nonzero reciprocal arguments or one
zero reciprocal argument, but not both.  The two possible initial parities
therefore contribute at most two solutions in total.

Finally, we consider the degenerate cases $s=0$ or $t=0$, for which the
preceding normalizations are not available.
If $s=0$, the output residue difference for an initial input of parity zero
is $(a-1)\iota(X)$.  Since $a\ne1$ and $\iota$ is a permutation, each
prescribed residue has exactly one preimage.  The same conclusion holds for
an initial input of parity one when $t=0$, because the output residue
difference is then $(1-a)\iota(X)$.  If $s=0$ and $c\ne0$, then
$t=h-c=-2c\ne0$, so the preceding argument shows that inputs of parity one
contribute at most one solution.  If $s=0$ and $c=0$, then $h=0$ and $t=0$,
so each initial parity contributes at most one solution.  The argument is
symmetric when $t=0$.
Combining the two initial parities and all the exceptional cases, we
conclude that every prescribed pair $(\eta,b)$ receives at most two
solutions.
\end{proof}

\begin{theorem}\label{T:reciprocal-apn}
The function $\mathcal B_{a,c,p}$ in \eqref{E:reciprocal-map} is an APN permutation of
$\Z_{2p}$ if and only if
\begin{equation}
 \chi(a)=\chi(-3)=\chi(1-4a)=\chi(1-4/a)=-1.       \label{E:inverse-condition}
\end{equation}
In particular, the APN property is independent of the translation parameter
$c$.
\end{theorem}

\begin{proof}
 According to the definition of differential uniformity in \cref{D:APN},
 after establishing the permutation property, we need to determine when the differential equation
 $\mathcal B_{a,c,p}(z+d)-\mathcal B_{a,c,p}(z)=v$ has at most two solutions
 $z\in\Z_{2p}$ for every $d\in\Z_{2p}\setminus\{0\}$ and
 $v\in\Z_{2p}$.  We first determine when $\mathcal B_{a,c,p}$ is a
 permutation.  We then prove the differential bound in the general cases
 where the parity component of $d$ is zero or one.  Finally, for each failed
 discriminant condition, we construct a DDT entry with at least three
 solutions.

We first consider the permutation property. If $\chi(a)=1$, fix a nonzero output residue $y$.  The unique inputs in the
cases $e=0$ and $e=1$ that give residue $y$ satisfy $x=y^{-1}$ and
$x+c=a/y$, respectively.  When $e=0$, inversion preserves quadratic
character, so the output parity is $\rho(x)=\rho(y)$.  When $e=1$,
$x+c\ne0$ and the correction term $1-\chi(x+c)^2$ vanishes.  Moreover,
$\chi(x+c)=\chi(a)\chi(y)=\chi(y)$, so its output parity is again $\rho(y)$.
The two inputs have different parities, whereas their outputs have the same
parity $\rho(y)$ and the same residue $y$ modulo $p$.  By the Chinese
remainder theorem, the two outputs coincide modulo $2p$, and hence
\eqref{E:reciprocal-map} is not a permutation.  Therefore $\chi(a)=-1$ is
necessary; under this assumption,
\cref{P:reciprocal-permutation} gives a permutation.

 We next prove the sufficiency of the four character conditions. If all four characters in \eqref{E:inverse-condition} are $-1$, the even and
odd bounds in \cref{L:reciprocal-even,L:reciprocal-odd} show that every DDT
entry with a nonzero input difference has multiplicity at most two: a nonzero
difference is either $(0,h)$ with $h\ne0$, or $(1,h)$ with arbitrary $h$.
Since a permutation has differential uniformity at least two, $\mathcal B_{a,c,p}$ is APN.

 It remains to consider the special cases in which one of the three
 discriminant conditions fails and thereby prove their necessity.
Assume $\chi(a)=-1$. The three remaining quantities in
\eqref{E:inverse-condition} are nonzero. Indeed, $p>3$, while
$1-4a=0$ or $1-4/a=0$ would imply $a=1/4$ or $a=4$, respectively, both
of which are squares. Suppose, contrary to the desired conclusion, that
one of these three quantities is a square.
 The corresponding discriminants occur in the equations derived in
\cref{L:reciprocal-even}. The relevant parity data and quadratic equations
are summarized in \cref{Tab:reciprocal-discriminants}.
\begin{table}[htbp]
\centering
\caption{Parity data for the three discriminant cases}
\label{Tab:reciprocal-discriminants}
\footnotesize
\begin{tabular}{|c|c|c|c|c|}
\hline
\text{Discriminant}
& \shortstack{\text{Input parity}\\\text{when }$X=0$}
& \shortstack{\text{Input parity}\\\text{for the two roots}}
& \shortstack{\text{Quadratic}\\\text{equation}}
& \shortstack{\text{Value of}\\$\kappa=z(z+1)$ }
\\ \hline
$-3$    & $0$ & $0$ & $z^2+z+1=0$      & $-1$      \\ \hline
$1-4a$  & $0$ & $1$ & $z^2+z+a=0$      & $-a$      \\ \hline
$1-4/a$ & $1$ & $0$ & $z^2+z+a^{-1}=0$ & $-a^{-1}$ \\
\hline
\end{tabular}
\end{table}
We now construct three inputs contributing to the same DDT entry.
Fix the row whose discriminant is a square.  Let $e$ and $e'$ be the input
parities in the second and third columns, respectively, and let
$\kappa=z(z+1)$ be the value in the last column.
Choose $k\in\F_p^*$ such that
$\chi(k)=-\chi(\kappa)\chi(a)^e$. Such a $k$ exists because each of the
two character values $1$ and $-1$ is attained by exactly $(p-1)/2$
elements of $\F_p^*$.
Consider the DDT
entry with input difference $(0,k)$, output residue difference $b=a^e/k$,
and output parity difference $\eta=(1-\chi(\kappa))/2$.

To verify this claim, first consider the exceptional value $X=0$ with
input parity $e$. It gives the first solution: its output residue
difference is $b$, while its output parity difference is
$\rho(a^e/k)=(1+\chi(a)^e\chi(k))/2=\eta$, where we have used
$\chi(k^{-1})=\chi(k)$.
The other two solutions arise from the quadratic equation in the selected
row of \cref{Tab:reciprocal-discriminants}. Since its discriminant is a
nonzero square, this equation has two distinct roots. Moreover, both roots
lie outside $\{0,-1\}$ because the equation has the form
$z^2+z+C=0$ with $C\ne0$. For either root $z$, set $X=kz$ and take the
input parity to be $e'$. Equation \eqref{E:even-normalized} then gives
$bk=-a^{e'}/(z(z+1))$. Since every row of
\cref{Tab:reciprocal-discriminants} satisfies
$-a^{e'}/\kappa=a^e$, where $\kappa=z(z+1)$, both roots yield the output
residue difference $b=a^e/k$. Their output parity difference is also
$(1-\chi(z(z+1)))/2=\eta$.
The two roots give distinct nonzero values of $X$. If $e=e'$, these values
are distinct from the exceptional value $X=0$; if $e\ne e'$, the
corresponding inputs are already distinguished by their parities.
Consequently, the exceptional solution $X=0$ and the two quadratic roots
give three distinct inputs contributing to the same DDT entry. This entry
therefore has multiplicity at least three, contradicting the APN property.
Hence all three discriminants must be nonsquares.

 Combining the sufficiency and necessity arguments proves the stated
 criterion.  Finally, we verify that it is independent of the translation
parameter $c$. Although $c$ occurs in the substitutions
$X=x+ec$, $s=h+c$, and $t=h-c$, it disappears from the normalized
equations and their discriminants. Hence the character conditions in
\eqref{E:inverse-condition} do not depend on $c$. This completes the proof.
\end{proof}

\subsection{Existence and parameter count}

Define $ \cA_p=\{a\in\F_p^*: \chi(a)=\chi(1-4a)=\chi(1-4/a)=-1\}$.
By \cref{T:reciprocal-apn}, the elements of this set give APN permutations precisely
when the additional condition $\chi(-3)=-1$, or equivalently
$p\equiv5\pmod6$, holds.

\begin{theorem}\label{T:parameter-count}
Let $T_p=\sum_{z\in\F_p}\chi(z(z-4)(1-4z))$.  Then we have
\begin{equation}
 |\cA_p|=\frac{p+1+T_p}{8},\qquad |T_p|\le2\sqrt p. \label{E:Ap-count}
\end{equation}
In particular, $|\cA_p|\ge(p+1-2\sqrt p)/8>0$ for $p\ge11$.
\end{theorem}

\begin{proof}
When $\chi(a)=-1$, the condition $\chi(1-4/a)=-1$ is equivalent to
$\chi(a-4)=1$, because $\chi(1-4/a)=\chi(a-4)\chi(a^{-1})=-\chi(a-4)$.
Away from $0$, $4$, and $1/4$, each of the three factors
$1-\chi(z)$, $1+\chi(z-4)$, and $1-\chi(1-4z)$ is either zero or two, and
their product equals eight exactly when
$\chi(z)=-1$, $\chi(z-4)=1$, and $\chi(1-4z)=-1$.  The three exceptional
elements are distinct, and hence
\begin{equation}
 8|\cA_p|=\sum_{z\in\F_p}
 (1-\chi(z))(1+\chi(z-4))(1-\chi(1-4z)).             \label{E:indicator-sum}
\end{equation}
At each exceptional root, the product on the right is zero, so no correction
term is required: the third factor vanishes at $z=0$, whereas the first factor
vanishes at $z=4$ and $z=1/4$.  With
$A=\chi(z)$, $B=\chi(z-4)$, and $C=\chi(1-4z)$, the product expands as
\begin{equation*}
 (1-A)(1+B)(1-C)=1+B-C-BC-A-AB+AC+ABC.
\end{equation*}
The linear character sums vanish because every nonconstant affine change of variable permutes $\F_p$.
Applying \cref{L:quadratic-character-sum} to the three square-free
quadratics gives $\sum_{z\in\F_p}\chi(z(z-4))=-1$,
$\sum_{z\in\F_p}\chi((z-4)(1-4z))=-\chi(-1)$, and
$\sum_{z\in\F_p}\chi(z(1-4z))=-\chi(-1)$.
With the signs in \eqref{E:indicator-sum}, the last two sums cancel and the
first contributes $1$.  The constant term contributes $p$, and the cubic term
is $T_p$.  Hence $8|\cA_p|=p+1+T_p$, which proves the equality in
\eqref{E:Ap-count}.  The cubic has three distinct roots, so the Weil bound
\eqref{E:weil} gives $|T_p|\le2\sqrt p$.
\end{proof}

\begin{corollary}\label{C:reciprocal-infinite}
If $p\equiv5\pmod6$ and $p>5$, then every $a\in\cA_p$ and $c\in\F_p$ gives an
APN permutation $\mathcal B_{a,c,p}$ of $\Z_{2p}$.  For each such order, there
are exactly $p|\cA_p|=p(p+1+T_p)/8$ distinct permutations within the
parametrized family, and this number is at least
$p(p+1-2\sqrt p)/8$.
\end{corollary}

\begin{proof}
Quadratic reciprocity gives $\chi(-3)=-1$ when $p\equiv5\pmod6$.  The result
follows from \cref{T:reciprocal-apn,T:parameter-count}.  The residue modulo
$p$ of the unique odd input whose output residue is zero equals $-c$.  Hence
the equality $\mathcal B_{a,c,p}=\mathcal B_{a',c',p}$ first gives
$c=c'$.  Evaluating both permutations at any odd input with field coordinate
$x+c\ne0$ then gives $a/(x+c)=a'/(x+c)$, and hence $a=a'$.  Thus the
parameterization is injective and the stated functions are distinct.
\end{proof}

\begin{corollary}\label{C:minus-one}
Let $p\ge11$ be a prime. If $p\equiv23\pmod{60}$ or $p\equiv47\pmod{60}$,
then, for every $c\in\F_p$, the function
$\mathcal B_{-1,c,p}\colon\Z_{2p}\to\Z_{2p}$ defined by
\[
\mathcal B_{-1,c,p}(z)
\equiv pR_{c,p}(z)
 +(p+1)(-1)^z(z+cE(z))^{p-2}
\pmod{2p}
\]
is an APN permutation.
\end{corollary}

\begin{proof}
The congruences $p\equiv23\pmod{60}$ or $p\equiv47\pmod{60}$ are equivalent to
$\chi(-1)=\chi(-3)=\chi(5)=-1$.  Indeed,
$\chi(-1)=-1$ is equivalent to $p\equiv3\pmod4$, while
$\chi(-3)=-1$ is equivalent to $p\equiv5\pmod6$; together these give
$p\equiv11\pmod{12}$.  Since $5\equiv1\pmod4$, quadratic reciprocity gives
$\chi(5)=\left(\frac{p}{5}\right)$, which equals $-1$ precisely when
$p\equiv2$ or $3\pmod5$.  The Chinese remainder theorem combines these
conditions into $p\equiv47$ or $23\pmod{60}$.  Since
$1-4(-1)=1-4/(-1)=5$, the condition in \cref{T:reciprocal-apn} holds for
$a=-1$.  Moreover, $A_{-1}(z)=(-1)^z$.
\end{proof}

\begin{example}\label{Ex:reciprocal-262}
Let $p=131$, $a=2$, and $c=0$.  The four characters in
\eqref{E:inverse-condition} are $ \chi(2)=\chi(128)=\chi(124)=\chi(130)=-1$.
For a direct reproduction, given $z\in\{0,\ldots,261\}$, put
$E=(1-(-1)^z)/2$, $X=z\bmod131$, $Y=(1+E)X^{129}\bmod131$, and
$R=\rho(X)+E(1-\chi(X)^2)\bmod2$.  Then
$ \mathcal B_{2,0,131}(z)=131R+132Y\pmod{262}$,
where the representative in $\{0,\ldots,261\}$ is used.  This is an APN
permutation of $\Z_{262}$.  Its differential spectrum is represented by
$(A_0,A_1,A_2)=(33541,1300,33541)$; the truth table is given in
Section~\ref{A:TruthTables}.
\end{example}

For the permutations in \eqref{E:reciprocal-map}, Table~\ref{Tab:reciprocal-ddt}
exhausts every
$a\in\cA_p$ and $c\in\F_p$ for six small primes, all congruent to $5$ modulo
$6$.  Duplicates were removed by comparing the truth tables, and each
row has $p|\cA_p|$ distinct permutations.  In the last column,
$m:(A_0,A_1,A_2)$ means that $m$ permutations have the indicated differential spectrum.

\begin{table}[htbp]
\caption{Differential spectra of completed-reciprocal permutations at small orders}
\label{Tab:reciprocal-ddt}
\centering
\scriptsize
\begin{tabularx}{\textwidth}{|c|c|c|r|c|>{\raggedright\arraybackslash}X|}
\hline
$p$&$n$&$|\cA_p|$&Permutations&No. of spectra&Permutations: differential spectrum\\ \hline
$11$&$22$&$2$&$22$&$3$&$10:(141,180,141)$; $10:(145,172,145)$; $2:(181,100,181)$\\ \hline
$17$&$34$&$2$&$34$&$2$&$32:(403,316,403)$; $2:(545,32,545)$\\ \hline
$23$&$46$&$3$&$69$&$3$&$33:(705,660,705)$; $33:(709,652,709)$; $3:(925,220,925)$\\ \hline
$29$&$58$&$4$&$116$&$2$&$112:(1207,892,1207)$; $4:(1625,56,1625)$\\ \hline
$41$&$82$&$4$&$164$&$2$&$160:(2443,1756,2443)$; $4:(3281,80,3281)$\\ \hline
$47$&$94$&$5$&$235$&$3$&$115:(3129,2484,3129)$; $115:(3133,2476,3133)$; $5:(4141,460,4141)$\\ \hline
\end{tabularx}
\end{table}

\begin{proposition}\label{P:no-polynomial}
Assume $\chi(a)=-1$.  The function $\mathcal B_{a,c,p}$ is not a polynomial
function over $\Z_{2p}$.
More generally, it cannot have the form
$P(z)Q(z)^{-1}$ with $P,Q\in\Z_{2p}[X]$ and
$Q(z)\in\Z_{2p}^*$ for every $z\in\Z_{2p}$.
\end{proposition}

\begin{proof}
Let $z,z'\in\Z_{2p}$ satisfy $z\equiv z'\pmod p$.  Polynomial evaluation
commutes with reduction modulo $p$, so $P(z)\equiv P(z')\pmod p$ and
$Q(z)\equiv Q(z')\pmod p$.  Since every $Q(z)$ is a unit modulo $2p$, its
reduction modulo $p$ is nonzero and may be inverted.  Thus every expression
$P(z)Q(z)^{-1}$ of the stated form has the same residue modulo $p$ at $z$ and
$z'$; the polynomial case is obtained by taking $Q=1$.

Since $\chi(a)=-1$, we have $a\ne1$.  Choose
$x\in\F_p\setminus\{0,-c,c/(a-1)\}$; such an $x$ exists because $p\ge11$.
The Chinese remainder theorem gives exactly two lifts of $x$ to
$\Z_{2p}$, one even and one odd.  Their outputs under
$\mathcal B_{a,c,p}$ reduce modulo $p$ to $x^{-1}$ and
$a(x+c)^{-1}$, respectively; both inverses exist by the choice of $x$.
Equality of these two residues would give $x+c=ax$, and hence
$x=c/(a-1)$, again excluded.  The two lifts therefore have different output
residues modulo $p$, although they have the same input residue modulo $p$.
This contradicts the dependence established in the first paragraph for every
function of the stated form.
\end{proof}

The completed reciprocal $x^{p-2}$ is itself an APN power function on prime
fields with $p\equiv2\pmod3$~\cite{BudaghyanPalDCC2025}.  After a nonzero
input difference is normalized, the only case involving a zero denominator
leads to a quadratic equation with discriminant $-3$, which is then a
nonsquare.  In \eqref{E:reciprocal-map}, the formula for the completed
reciprocal depends on the input parity, and the output parity is determined by
the quadratic character.  By \cref{P:no-polynomial}, the resulting permutation is not a
polynomial function on the residue class ring.

\subsection{Orders different from those of the standard classes}

\begin{theorem}\label{T:reciprocal-orders}
The arithmetic progression $p\equiv11027\pmod{44100}$ contains infinitely
many primes. For every such prime $p$, put $n=2p$. Then none of the integers
$n-1,n+1,n+2,n+3,n+4$ is a prime power. Moreover, for every $c\in\F_p$,
the permutation $\mathcal B_{-1,c,p}$ is APN. Consequently, $n$ is neither of the
form $r+1$ nor of any of the forms $r-j$ with $1\le j\le4$, where $r$ is a
prime power. Thus these orders lie outside the order forms of the standard
finite-field and Singer classes listed in
\cref{Tab:intro-comparison}.
\end{theorem}

\begin{proof}
Since $44100=\operatorname{lcm}(60,9,25,49)$ and
$\gcd(11027,44100)=1$, Dirichlet's theorem guarantees infinitely many primes
in this arithmetic progression. Reducing the defining congruence modulo the
four indicated divisors gives $p\equiv47\pmod{60}$,
$p\equiv2\pmod9$, $p\equiv2\pmod{25}$, and
$p\equiv2\pmod{49}$.

The congruence $p\equiv47\pmod{60}$ implies that
$\mathcal B_{-1,c,p}$ is APN for every $c\in\F_p$ by
\cref{C:minus-one}. The other three congruences give
$n-1=2p-1\equiv3\pmod9$, $n+1=2p+1\equiv5\pmod{25}$, and
$n+3=2p+3\equiv7\pmod{49}$. Hence $v_3(n-1)=1$,
$v_5(n+1)=1$, and $v_7(n+3)=1$. Each of $n-1,n+1,n+3$ is larger than
the corresponding prime, so each has another prime divisor and therefore
cannot be a prime power.
Furthermore, $p\equiv47\pmod{60}$ implies $3\mid(p+1)$. Thus
$n+2=2(p+1)$ has the two distinct prime divisors $2$ and $3$. Finally,
$p$ is odd, so $p+2>1$ is odd and has an odd prime divisor. Together with
the factor $2$, this shows that $n+4=2(p+2)$ also has at least two distinct
prime divisors. Therefore, none of $n-1,n+1,n+2,n+3,n+4$ is a prime power,
which proves the result.
\end{proof}

\begin{example}\label{Ex:direct-separated}
The first term $p=11027$ in the progression of
\cref{T:reciprocal-orders} is prime. Take $a=-1$ and $c=0$ in
\cref{Con:reciprocal}. More explicitly, for
$z\in\{0,\ldots,22053\}$, set $E=(1-(-1)^z)/2$ and
$X=z\bmod11027$, and define $Y$ and $R$ by
$Y\equiv(1-2E)X^{11025}\pmod{11027}$ and
$R\equiv\rho(X)+E(1-\chi(X)^2)\pmod2$. Then the function defined by
$F(z)\equiv11027R+11028Y\pmod{22054}$ is the APN permutation
$\mathcal B_{-1,0,11027}$ supplied by \cref{C:minus-one}.
For $n=22054$, the integers $n-1,n+1,n+2,n+3,n+4$ have the
factorizations listed in \cref{Tab:direct-separated-factorizations}.

\begin{table}[htbp]
\centering
\caption{Factorizations of the neighboring integers for $n=22054$}
\label{Tab:direct-separated-factorizations}
\footnotesize
\begin{tabular}{|c|c|}
\hline
\text{Integer} & \text{Factorization} \\
\hline
$n-1$ & $3\cdot7351$ \\
\hline
$n+1$ & $5\cdot11\cdot401$ \\
\hline
$n+2$ & $2^3\cdot3\cdot919$ \\
\hline
$n+3$ & $7\cdot23\cdot137$ \\
\hline
$n+4$ & $2\cdot41\cdot269$ \\
\hline
\end{tabular}
\end{table}
\end{example}

\section{Comparison with Known Classes}\label{S:Comparison}

Table~\ref{Tab:intro-comparison} summarizes the orders of the known and new
classes.  The permutations in Section~\ref{S:Singer} use fixed points of
projectivities of a line.  For the binomials in Section~\ref{S:Binomial},
for each scalar in $\F_q^*$, at most one projective point satisfies the
corresponding equation.  The permutations in
Section~\ref{S:Reciprocal} distinguish the roots of the resulting quadratic
equations by their quadratic characters.  The parameters are, respectively, elements
$M\in\PGL_2(q)$ satisfying $MJ\ne JM$, binomials satisfying the permutation
and greatest-common-divisor conditions in
\cref{Con:binomial,T:exponent-condition}, and triples $(p,a,c)$ satisfying
\eqref{E:inverse-condition}.  The first two constructions arise from linear
and projective actions over finite fields, whereas the third is described by
elementary congruence and character data.  More importantly, the proofs use
three different bounds:
at most two fixed points; at most one projective point for each scalar in
$\F_q^*$; and at most two roots of the equations involving reciprocals after
their quadratic characters are taken into account.  Consequently, none of the
three proofs is a specialization of another proof in the paper; this statement
does not assert that the resulting families are disjoint at a common order.

\begin{corollary}\label{C:order-separation}
Each of the three constructions contains an infinite subfamily of composite
orders $n$ that do not have the forms $r-1,r-2,r-3,r-4$, where $r$ is a
prime power.  These forms
occur in the Welch--Costas, Panario--Sakzad--Stevens--Wang, and standard
Golomb classes.  The constructions
in Sections~\ref{S:Binomial} and \ref{S:Reciprocal} contain infinitely many
orders not of the form $r+1$ occurring in the Singer construction, while the
subfamily in \cref{T:singer-separated} contains no order of the form $2p$.
\end{corollary}
\begin{proof}
The Welch--Costas and Panario--Sakzad--Stevens--Wang orders have the form
$r-1$, and the three standard Golomb orders have the forms $r-2,r-3,r-4$,
where $r$ is a prime power. Thus an order $n$ belongs to one of these
classes only if one of $n+1,n+2,n+3,n+4$ is a prime power.
\cref{T:singer-separated,T:ternary-separated,T:reciprocal-orders} each gives
an infinite subfamily for which none of these four integers is a prime
power.

The orders in \cref{T:singer-separated} are divisible by both $3$ and $4$.
In \cref{T:ternary-separated}, the integer $d\equiv14\pmod{4620}$ is even,
and $(3^d-1)/2=(3^{d/2}-1)(3^{d/2}+1)/2$ is composite. The orders in
\cref{T:reciprocal-orders} have the form $2p$, where $p$ is an odd prime.
Hence all three subfamilies consist of composite orders.

The Singer construction has order $r+1$, so it requires $n-1=r$ to be a
prime power; \cref{T:ternary-separated,T:reciprocal-orders} also exclude
this possibility. Finally, \cref{T:singer-separated} states directly that
its orders are not of the form $2p$. These are exactly the assertions of
the corollary.
\end{proof}

\section{Conclusion}\label{S:Conclusion}

In this paper, we presented three infinite classes of APN permutations on
$\Z_n$. The Singer-cycle construction produced APN permutations on
$\Z_{q+1}$ for every odd prime power $q$. For a fixed generator $R$ of a
Singer group and a fixed base point, the $q(q^2-1)-2(q+1)$ admissible
projectivities induced distinct permutations. The binomial construction
gave differential uniformity at most $q-1$ under the exponent condition
$\gcd(r-t,q^d-1)=q-1$. For $q=3$, invertible ternary linearized polynomials
yielded an explicit infinite APN subfamily on $\Z_{(3^d-1)/2}$.
The completed-reciprocal construction produced a parametrized family
$\mathcal B_{a,c,p}$ on $\Z_{2p}$. The permutation $\mathcal B_{a,c,p}$ was APN
permutation if and only if
$\chi(a)=\chi(-3)=\chi(1-4a)=\chi(1-4/a)=-1$. For every prime $p>5$ with
$p\equiv5\pmod6$, the associated character sum gave the exact number of APN
permutations in this family. The three APN proofs relied, respectively, on
fixed points of projectivities, numbers of projective solutions, and
quadratic characters. The small-order computations confirmed the
theoretical results and determined the differential spectra in the stated
parameter ranges.
Each construction also contained an infinite subfamily of composite orders
outside the standard forms $r-1,r-2,r-3,r-4$, where $r$ is a prime power.
These are the order forms associated with the Welch--Costas, standard Golomb,
and Panario--Sakzad--Stevens--Wang classes. To the best of our knowledge,
the three constructions presented here are the first infinite classes
reported since the Panario--Sakzad--Stevens--Wang class was introduced in
2011 that provide APN permutations on $\Z_n$ for infinitely many composite
orders $n$ outside these standard forms.

\section{Appendix}\label{S:Appendix}

\subsection{Truth Tables of Representative APN Permutations}
\label{A:TruthTables}

This appendix gives one moderate-order representative from each construction,
together with its truth table $[F(0),F(1),\ldots,F(n-1)]$ and differential
spectrum $(A_0,A_1,\ldots,A_n)$, where $A_j$ is defined in
\eqref{E:spectrum}. All examples were generated and verified by Magma V2.28-3.

\lstset{basicstyle=\ttfamily\scriptsize,breaklines=true,
breakatwhitespace=false,columns=fullflexible,frame=single,
literate={,}{{,\allowbreak}}1}

\subsubsection{Example induced by a Singer cycle on \texorpdfstring{$\Z_{258}$}{Z258}}
Let $f(X)=X^2-6X+3\in\F_{257}[X]$, let
$K=\F_{257}[X]/(f(X))$, and let $\alpha=X+(f(X))$.  The discriminant $24$ is
a nonsquare in $\F_{257}$, so $f$ is irreducible.  Moreover,
$\alpha^{33024}=-1$, $\alpha^{22016}=179+240\alpha$, and
$\alpha^{1536}=125+176\alpha$.  Since the prime divisors of
$|K^*|=66048$ are $2,3,43$, the element $\alpha$ is primitive.

Represent elements of $K$ by column vectors in the ordered basis
$(1,\alpha)$, put $e_0=(1,0)^{\mathsf T}$, and take
\begin{equation*}
 R=\begin{pmatrix}0&254\\1&6\end{pmatrix},~~ M=\begin{pmatrix}1&1\\0&1\end{pmatrix},~~\text{and}~~ P_0=\infty.
\end{equation*}
The relation $\alpha^2=6\alpha+254$ shows directly that $R$ is the matrix of
multiplication by $\alpha$.  For $0\le i\le257$, put
$P_i=[R^ie_0]=[\alpha^i]$.  These are all the points of
$\mathbb P^1(\F_{257})$.  For each $i$, define $F(i)$ as the unique $j$ such
that $MR^ie_0=\lambda R^je_0$ for some $\lambda\in\F_{257}^*$.  Thus the
specified matrices determine the following truth table without any
unspecified field choices.
The induced permutation has $\du=2$, and its differential spectrum is
$(A_0,A_1,A_2)=(33025,256,33025)$.
\begin{lstlisting}
F = [0,143,232,83,145,21,208,215,26,116,198,237,203,129,2,200,213,239,5,241,194,247,55,114,219,204,256,195,172,25,141,212,17,121,18,94,170,67,171,234,47,163,16,251,190,242,227,74,133,199,252,205,183,207,233,246,186,59,243,209,248,72,162,231,238,167,43,105,214,45,188,60,202,112,10,206,56,149,79,73,4,124,106,58,118,51,230,220,222,62,102,193,97,153,52,217,196,150,123,85,113,111,33,32,44,165,84,88,161,181,127,146,147,254,34,257,87,228,122,54,135,15,24,138,49,210,6,125,13,245,1,148,19,131,177,160,140,225,92,80,174,90,249,9,235,158,185,157,93,8,69,99,176,191,36,29,168,46,37,107,71,166,68,78,223,110,120,22,117,81,151,142,20,159,152,255,12,89,119,180,95,31,3,30,211,179,197,98,14,108,96,221,48,28,11,57,169,40,187,201,175,63,182,236,139,50,164,173,53,134,66,218,101,189,154,192,41,42,61,7,27,100,104,23,144,156,155,77,75,103,65,38,250,229,136,35,91,253,86,126,224,226,216,137,70,130,82,64,184,115,109,39,132,240,178,76,244,128].
\end{lstlisting}

\subsubsection{Binomial example on
\texorpdfstring{$\Z_{364}$}{Z364}}
Let $f(X)=X^6+X^5+X^3-1\in\F_3[X]$, $K=\F_3[X]/(f(X))$, and $\alpha=X+(f(X))$,
where $f$ is primitive (and hence irreducible); thus $\alpha$ has order $3^6-1=728$.
Take $d=6$, $s=1$, $c=-\alpha$, and $k=1$, so
$H(x)=x^3-\alpha x$.  Use column vectors in the ordered power basis
$(1,\alpha,\ldots,\alpha^5)$ and put $e_0=(1,0,0,0,0,0)^{\mathsf T}$.
Multiplication by $\alpha$ and the linearized transformation $L=H$ are represented by
\begin{equation*}
 R=\begin{pmatrix}0&0&0&0&0&1\\1&0&0&0&0&0\\0&1&0&0&0&0\\0&0&1&0&0&2\\0&0&0&1&0&0\\0&0&0&0&1&2\end{pmatrix}
~~\text{and}~~
 L=\begin{pmatrix}1&0&1&1&1&1\\2&0&0&1&2&1\\0&2&0&2&0&1\\0&1&1&0&0&0\\0&0&0&1&2&1\\0&0&2&0&0&1\end{pmatrix}.
\end{equation*}
For $0\le i\le363$, define $F(i)$ as the unique $j$ satisfying
$LR^ie_0\in\{R^je_0,-R^je_0\}$.  Since $R^{364}=-I$, the vectors
$R^ie_0$, up to sign, give all the points of $K^*/\F_3^*$.  This rule
generates the following truth table.
The induced permutation has $\du=2$, and its differential spectrum is
$(A_0,A_1,A_2)=(33124,65884,33124)$.
\begin{lstlisting}
F = [259,261,52,139,214,153,201,33,50,141,325,275,151,338,92,67,303,236,190,286,96,316,22,147,152,29,56,93,314,244,60,271,94,87,42,86,197,331,283,136,175,273,361,114,198,119,229,178,216,298,341,174,323,203,248,134,127,54,27,285,81,319,217,252,112,63,208,282,74,125,327,89,329,122,84,181,296,165,301,156,276,64,145,211,180,11,1,142,253,177,237,98,82,297,71,279,362,326,258,194,347,123,70,65,255,113,131,224,66,359,262,108,103,118,263,79,41,4,36,158,182,44,88,315,12,352,235,102,339,251,132,227,334,15,354,302,39,320,343,230,167,10,350,281,95,110,163,221,53,292,126,328,155,146,231,238,144,166,242,278,270,13,137,5,35,45,6,14,199,239,159,140,353,78,299,91,124,3,202,240,274,311,226,228,317,284,254,220,25,150,121,333,116,31,186,209,293,257,76,72,115,109,83,219,99,360,8,340,268,250,348,345,264,277,90,256,62,191,363,309,260,249,75,148,176,337,40,157,138,225,128,184,213,172,69,342,101,193,324,97,218,246,189,330,290,68,48,294,266,307,349,173,32,21,30,188,289,0,162,73,59,205,19,28,85,313,164,232,304,344,265,61,291,80,100,243,187,267,160,23,37,210,245,183,106,322,206,355,223,358,247,154,196,43,171,49,215,168,16,310,169,222,2,335,77,207,51,161,192,269,280,34,357,117,336,143,104,26,356,47,305,24,308,195,111,17,346,133,185,55,312,272,321,332,149,306,130,204,351,318,295,58,57,300,234,18,212,120,170,241,9,38,288,105,233,287,107,20,7,179,135,200,129,46].
\end{lstlisting}

\subsubsection{Example from completed reciprocals on
\texorpdfstring{$\Z_{262}$}{Z262}}
Take $p=131$, $a=2$, and $c=0$. For each
$z\in\{0,\ldots,261\}$, set $E=(1-(-1)^z)/2$,
$X=z\bmod131$, and $Y=(1+E)X^{129}\bmod131$, and let
$R\in\{0,1\}$ be determined by
$R\equiv\rho(X)+E(1-\chi(X)^2)\pmod2$. The corresponding truth-table entry
is the unique representative in $\{0,\ldots,261\}$ satisfying
$F(z)\equiv131R+132Y\pmod{262}$. Here $X^{129}=0$ when $X=0$, as required
by the completed reciprocal. The four character arguments are
$2,128,124,130$, and direct exponentiation in $\F_{131}$ gives
$\chi(2)=\chi(128)=\chi(124)=\chi(130)=-1$. The resulting permutation has
$\du=2$ and differential spectrum
$(A_0,A_1,A_2)=(33541,1300,33541)$.

\begin{lstlisting}
F = [0,133,66,219,33,79,22,19,82,73,118,155,11,111,234,201,41,108,182,138,59,181,6,114,202,173,126,199,117,244,214,186,86,139,27,161,91,170,231,37,226,163,78,259,3,67,225,184,101,115,76,36,63,47,148,31,124,46,61,171,107,247,243,235,43,127,2,176,210,38,204,48,242,140,239,145,50,97,42,136,113,241,8,60,39,74,32,256,198,237,246,203,178,62,177,80,116,208,258,221,169,227,9,28,228,141,110,251,205,119,106,144,193,29,77,180,96,187,10,240,250,157,160,229,56,87,26,196,218,261,130,131,1,132,175,197,105,44,75,164,233,236,143,153,121,206,35,213,54,102,200,249,25,12,188,142,21,252,165,103,122,166,224,172,135,185,15,51,216,69,215,190,147,156,195,137,99,57,92,71,123,152,18,257,89,34,81,248,154,253,151,83,189,93,183,217,129,4,88,158,150,146,24,222,70,85,7,100,245,84,68,95,55,16,30,209,168,64,128,134,53,230,167,94,162,223,40,232,104,254,45,207,179,149,14,194,5,220,191,17,125,212,72,255,211,23,90,192,159,20,120,238,13,58,49,112,109,52,98,174,65,260].
\end{lstlisting}

\end{document}